\documentclass[12pt]{article}
\usepackage[margin=1in]{geometry}

\usepackage{amsmath,amsthm,amssymb}
\usepackage{mathrsfs}
\usepackage[shortlabels]{enumitem}
\usepackage{stmaryrd}
\usepackage{shuffle}
\usepackage{multicol} 
\usepackage{mathtools}
\usepackage{commath}
\usepackage{caption}
\usepackage{subcaption}

\usepackage{tikz}
\usetikzlibrary{calc,fit,shapes,positioning}
\tikzstyle{vertex} = [circle,draw=black,minimum size=2pt,inner sep=1pt] 
\tikzstyle{vertexSmall} = [circle,draw=black,minimum size=2pt,fill=white] 
\tikzstyle{edge} = [line width=1pt] 
\tikzstyle{arrow} = [edge,->] 
\tikzstyle{descent} = [red,dashed,line width=2pt,->] 

\newcommand{\N}{\mathbb{N}}

\newcommand{\Z}{\mathbb{Z}}

\newcommand{\lp}{\left(}
\newcommand{\rp}{\right)}

\newcommand{\sm}{\setminus}

\DeclareRobustCommand*{\ora}{\overrightarrow}  
\renewcommand{\phi}{\varphi}

\DeclareMathOperator{\des}{des}

\DeclareMathOperator{\inv}{inv}

\usepackage{hyperref}
\hypersetup{
  colorlinks   = true, 
  urlcolor     = black, 
  linkcolor    = blue, 
  citecolor   = green 
} 
\usepackage{cleveref}
\newtheorem{theorem}{Theorem}[section]
\newtheorem{cor}[theorem]{Corollary}
\newtheorem{conjecture}[theorem]{Conjecture}
\newtheorem{prop}[theorem]{Proposition}
\newtheorem{lemma}[theorem]{Lemma}
\newtheorem{question}[theorem]{Question}
\newtheorem{prob}[theorem]{Problem}

\newtheorem{claim}[theorem]{Claim}
\theoremstyle{definition}
\newtheorem{definition}[theorem]{Definition}
\newtheorem{remark}[theorem]{Remark}

\newcommand{\A}{A}

\newcommand{\floor}[1]{\lfloor #1 \rfloor}

\newcommand{\sig}{\sigma}
\newcommand{\sub}{\subseteq}

\DeclareMathOperator{\mult}{mult}

\newcommand{\X}{\widetilde{X}}
\newcommand{\maj}{\mathrm{maj}}

\usepackage{tikz}
\usetikzlibrary{calc}

\begin{document}
\title{Eulerian Polynomials for Digraphs}
\author{Kyle Celano\footnote{Dept.\ of Mathematics, Wake Forest University \url{celanok@wfu.edu}} \and Nicholas Sieger\footnote{Dept.\ of Mathematics, University of California San Diego {\tt nsieger@ucsd.edu}}\and Sam Spiro\footnote{Dept.\ of Mathematics, Rutgers University {\tt sas703@scarletmail.rutgers.edu}. This material is based upon work supported by the National Science Foundation Mathematical Sciences Postdoctoral Research Fellowship under Grant No. DMS-2202730.}}
\date{\today}
\maketitle
\begin{abstract}
   Given an $n$-vertex digraph $D$ and a labeling $\sigma:V(D)\to [n]$, we say that an arc $u\to v$ of $D$ is a descent of $\sigma$ if $\sigma(u)>\sigma(v)$.   Foata and Zeilberger introduced a generating function $A_D(t)$ for labelings of $D$ weighted by descents, which simultaneously generalizes both Euleraian polynomials and Mahonian polynomials.  Motivated by work of Kalai, we look at problems related to $-1$ evaluations of $A_D(t)$.  In particular, we give a combinatorial interpretation of $|A_D(-1)|$ in terms of ``generalized alternating permutations'' whenever the underlying graph of $D$ is bipartite.
\end{abstract}

\section{Introduction}
Descents and inversions are two of the oldest and most well-studied \textit{permutation statistics} dating back to work of MacMahon 
\cite{macmahon2001combinatory,macmahon1913indices}. A \textit{descent} of a permutation $\sigma$ on the set $[n]:=\{1,2,\dots,n\}$ is an index $i\in [n-1]$ such that $\sigma(i)>\sigma(i+1)$, and an inversion is a pair  of integers $(i,j)$ with $1\leq i<j\leq n$ such that $\sigma(i)>\sigma(j)$. The number of descents and inversions of $\sigma$ are denoted by $\des(\sigma)$ and $\inv(\sigma)$, respectively. For example, if $\sigma=23154$ (written in \textit{one-line notation}, meaning $\sigma(1)=2,\ \sigma(2)=3\dots$), then $2$ and $4$ are descents whereas $(1,3),(2,3),(4,5)$ are inversions, so $\des(\sigma)=2$ and $\inv(\sigma)=3$.

If $\mathfrak{S}_n$ is the set of permutations of $[n]$, then the generating functions 
\[A_n(t)=\sum_{\sigma\in \mathfrak{S}_n}t^{\des(\sigma)}\quad M_n(t)=\sum_{\sigma\in \mathfrak{S}_n}t^{\inv(\sigma)}\]
are called the \textit{Eulerian} and \textit{Mahonian} polynomials respectively. Both of these polynomials are important objects of study in many branches of combinatorics and have been generalized in many different ways.  In this paper, we consider a polynomial due to Foata and Zeilberger~\cite{FOATA199679} which generalizes both the Eulerian and Mahonian polynomials via directed graphs.

A \textit{directed graph} or \textit{digraph} is a pair $D=(V,E)$ consisting of a finite set $V$ of \textit{vertices} and a subset $E\subset V\times V$ of \textit{(directed) edges} or \textit{arcs}. We will sometimes denote arcs $(u,v)\in E$ by $u\to v$ or even $uv$ where convenient. We further assume that $D$ has no loops i.e. no arcs of the form $(v,v)$. We do not allow multiple directed edges from one vertex to another, though it is easy to adapt our forthcoming definitions to accommodate this.

A \textit{permutation} of an $n$-vertex digraph $D=(V,E)$ is a bijection $\sigma:V\to [n]$. We will use the notation $\mathfrak{S}_D,\ \mathfrak{S}_V$, or $\mathfrak{S}_n$ to denote the set of permutations of $D$.  For a given directed graph $D=(V,E)$ and a permutation $\sigma$ of $D$, a \textit{$D$-descent} (or just \textit{descent} when $D$ is understood) is an arc $u\to v$ such that $\sigma(u)>\sigma(v)$. The total number of $D$-descents of a permutation $\sigma$ is denoted by $\des_D(\sigma)$; see \Cref{fig:basic example} for an example.

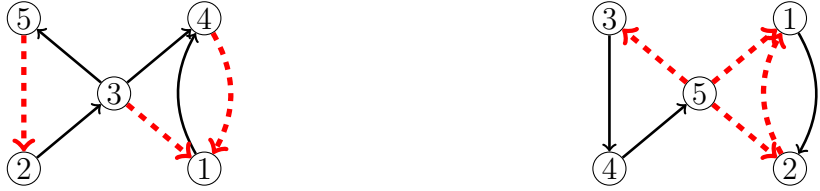
\begin{figure}[h]
	\centering
	\begin{subfigure}[b]{.5\textwidth}
		\centering
		\begin{tikzpicture}
			\node[vertex] (a) at (0,0) {$3$} ;
			\node[vertex] (b) at (1.2,1) {$4$} ;
			\node[vertex] (c) at (1.2,-1) {$1$} ;
			\node[vertex] (d) at (-1.2,1) {$5$} ;
			\node[vertex] (e) at (-1.2,-1) {$2$} ;
			
			\draw[arrow] (a) --(b);
			\draw[descent] (b) edge[bend left] (c);
			\draw[arrow] (c) edge[bend left] (b);
			\draw[descent] (a)--(c);
			\draw[arrow] (a)--(d);
			\draw[descent] (d) -- (e);
			\draw[arrow] (e) -- (a);
			
		\end{tikzpicture} 
	\end{subfigure}   
	~
	\begin{subfigure}[b]{.4\textwidth}
		\centering
		\begin{tikzpicture}
			\node[vertex] (a) at (0,0) {$5$} ;
			\node[vertex] (b) at (1.2,1) {$1$} ;
			\node[vertex] (c) at (1.2,-1) {$2$} ;
			\node[vertex] (d) at (-1.2,1) {$3$} ;
			\node[vertex] (e) at (-1.2,-1) {$4$} ;
			
			\draw[descent] (a) --(b);
			\draw[arrow] (b) edge[bend left] (c);
			\draw[descent] (c) edge[bend left] (b);
			\draw[descent] (a)--(c);
			\draw[descent] (a)--(d);
			\draw[arrow] (d) -- (e);
			\draw[arrow] (e) -- (a);

		\end{tikzpicture}
		\label{fig:basic example with digon}    
	\end{subfigure}  
	\caption{Two labelings $\pi:V(D)\to [5]$ where descent arcs are marked by red dashed lines.}
	\label{fig:basic example}
\end{figure}

\begin{remark}\label{eg:classic}We claim that the statistic $\des_D$ generalizes both $\des$ and $\inv$.  To see this, let $\ora{P}_n$ be the digraph with vertex set $[n]$ and with arcs $(i,i+1)$ for $i\in [n-1]$, and let $\ora{K}_n$ be the digraph with vertex set $[n]$ and with arcs $(i,j)$ for integers $1\leq i<j\leq n$. The reader can check that a \textit{$\ora{P}_n$-descent} is a \textit{descent} (in the classical meaning) and a \textit{$\ora{K}_n$-descent} is an inversion, and hence
\[\des_{\ora{P}_n}(\sigma)=\des(\sigma)\quad \des_{\ora{K}_n}(\sigma)=\inv(\sigma).\]
See \Cref{fig: desD is desinv} for an example.
\end{remark}

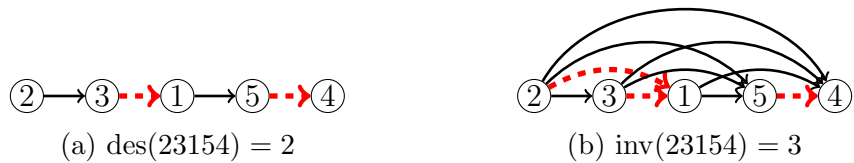
\begin{figure}[h]
    \centering
    \begin{subfigure}[b]{.4\textwidth}   
    \centering
    \begin{tikzpicture}

        \node[vertex] (a) at (0,0) {$2$} ;
        \node[vertex] (b) at (1,0) {$3$} ;
        \node[vertex] (c) at (2,0) {$1$} ;
        \node[vertex] (d) at (3,0) {$5$} ;
        \node[vertex] (e) at (4,0) {$4$} ;

        \draw[arrow] (a) --(b);
        \draw[descent] (b)--(c);
        \draw[arrow] (c)--(d);
        \draw[descent] (d) -- (e);

    \end{tikzpicture}
    \caption{ $\des(23154)=2$ }
    \label{fig:des Pn is des}    
    \end{subfigure}
    \begin{subfigure}[b]{.4\textwidth}
    \centering
    \begin{tikzpicture}
        \node[vertex] (a) at (0,0) {$2$} ;
        \node[vertex] (b) at (1,0) {$3$} ;
        \node[vertex] (c) at (2,0) {$1$} ;
        \node[vertex] (d) at (3,0) {$5$} ;
        \node[vertex] (e) at (4,0) {$4$} ;

        \draw[arrow] (a) --(b);
        \draw[descent] (a) edge[bend left] (c);
        \draw[arrow] (a) edge[bend left=45] (d);
        \draw[arrow] (a) edge[bend left=60] (e);
        \draw[descent] (b)--(c);
        \draw[arrow] (b) edge[bend left] (d);
        \draw[arrow] (b) edge[bend left=45] (e);
        \draw[arrow] (c)--(d);
        \draw[arrow] (c) edge[bend left] (e);
        \draw[descent] (d) -- (e);
        
    \end{tikzpicture}
    \caption{ $\inv(23154)=3$ }
    \label{fig:des Kn is inv}    
    \end{subfigure}    
    \caption{The Eulerian polynomial  $\A_D(t)$ captures both permutation descents and inversions.}
    \label{fig: desD is desinv}
\end{figure}
With all this in mind, we can now define the central object of study for this paper: the \textit{Eulerian polynomial} of a digraph $D=(V,E)$ is the generating function
\begin{equation}
    A_D(t)=\sum_{\sigma\in \mathfrak{S}_D}t^{\des_D(\sigma)}.
\end{equation}
In particular, the previous remark implies $A_{\ora{P_n}}(t)=A_n(t)$ and $A_{\ora{K_n}}(t)=M_n(t)$.

 \subsection{Main results}

The primary objective of this paper is to study evaluations of $A_D(t)$ at $-1$.   This is a problem in the area of \textit{combinatorial reciprocity}, which studies combinatorial polynomials evaluated at negative integers.  For example, the classical Eulerian and Mahonian polynomials both have good combinatorial interpretations for their evaluation at $-1$: the former being the number of \textit{alternating permutations} \cite{foata2005theorie} and the later being the number of \textit{correct proofs of the Riemann hypothesis}\footnote{As of the time of writing.}. Many more results on combinatorial reciprocity can be found in the book by Beck and Sanyal~\cite{beck2018combinatorial}.


In order to understand $A_D(-1)$ , we utilize the following key observation made by Kalai~\cite[Section 8.1]{KALAI2002412}.
\begin{prop}\label{prop:negOne}
    If $D,D'$ are orientations of the same graph $G$, then $|\A_D(-1)|=|\A_{D'}(-1)|$.
\end{prop}
This result follows from the fact that if any arc of $D$ is reversed, then the number of descents for any permutation $\sig\in \mathfrak{S}_V$ changes by exactly 1.  
With \Cref{prop:negOne} in mind, for any graph $G$ we can define
\[\nu(G):=|\A_D(-1)|,\]
where $D$ is any orientation of $G$.  The problem of studying $\nu(G)$ was first introduced by Kalai \cite{KALAI2002412} due to its relation with the Condorcet paradox in social choice theory, and a few basic properties of $\nu(G)$ were established by Even-Zohar \cite{Even-Zohar2017}.  Outside of this, nothing seems to be known about $\nu(G)$ despite Kalai raising the problem over 20 years ago.

In this paper, we prove three types of results related to $\nu(G)$: we give combinatorial interpretations for $\nu(G)$ for a large class of graphs $G$, we determine the maximum and minimum values achieved by $\nu(G)$ amongst $n$ vertex trees, and we consider the refined problem of determining the multiplicity of $-1$ as a root of $\A_D(t)$.

\subsubsection{Combinatorial Interpretations for $\nu(G)$}
A key observation towards understanding $\nu(G)$ is a result of Foata and Sch\"utzenberger~\cite{foata2005theorie} (see also \cite[Exercise 135]{stanley_2011}) which states that the Eulerian polynomial $A_n(t)$ evaluated at $t=-1$ is equal (up to sign) to the number of alternating permutations of size $n$, i.e.\ the number of permutations $\sig$ with $n$ odd and $\sig(1)<\sig(2)>\sig(3)<\cdots >\sig(n)$.  Because $A_n(t)=A_{\ora{P}_n}(t)$ for $\ora{P}_n$ the directed path, this result implies $\nu(P_n)$ is equal to the number of alternating permutations of size $n$. 

Given this observation, it is natural to expect $\nu(G)$ to count ``alternating permutations for graphs'' for some generalized notion of alternating permutations.  There are many such generalizations one could consider, for example, one could force every maximal path of $G$ to be an alternating permutation.  However, it turns out that the definition we will want to consider is the following (non-obvious) generalization.

\begin{definition}
    Given an $n$-vertex graph $G$, we say that an ordering $\pi=(\pi_1,\ldots,\pi_n)$ of the vertex set $V(G)$ is an \textit{even sequence} if each of the subgraphs $G[\pi_1,\ldots,\pi_i]$ induced by the first $i$ vertices of $\pi$ have an even number of edges for all $1\le i\le n$.  We let $\eta(G)$ denote the number of even sequences of $G$. 
\end{definition}

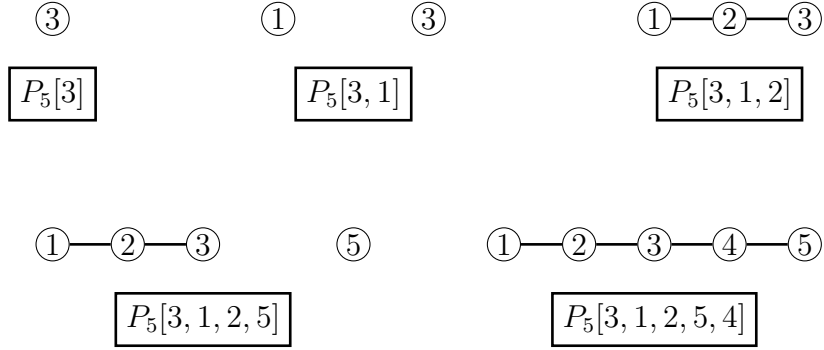
\begin{figure}[h]
    \centering
    \begin{tikzpicture}
        \coordinate (base1) at (-5,3);
        \coordinate (base2) at (0,3);
        \coordinate (base3) at (5,3);
        \coordinate (base4) at (-3,0);
        \coordinate (base5) at (3,0);

        \coordinate (p1) at (-2,0);
        \coordinate (p2) at (-1,0);
        \coordinate (p3) at (0,0);
        \coordinate (p4) at (1,0);
        \coordinate (p5) at (2,0);
        \coordinate (label) at (0,-1);
        
        \node[vertex] (13) at ($(base1) + (p3)$) {$3$};
        \node[draw,line width=1pt] at($(base1) + (p3) + (label)$) {$P_5[3]$};

        \node[vertex] (23) at ($(base2) + (p3)$) {$3$};
        \node[vertex] (21) at ($(base2) + (p1)$) {$1$};
        \node[draw,line width=1pt] at($(base2) + 0.5*(p1) + 0.5*(p3) + (label)$) {$P_5[3,1]$};
        
        \node[vertex] (33) at ($(base3) + (p3)$) {$3$};
        \node[vertex] (31) at ($(base3) + (p1)$) {$1$};
        \node[vertex] (32) at ($(base3) + (p2)$) {$2$};
        \draw[edge] (33) --(32) --(31);
        
        \node[draw,line width=1pt] at($(base3) + (p2) + (label)$) {$P_5[3,1,2]$};

        \node[vertex] (43) at ($(base4) + (p3)$) {$3$};
        \node[vertex] (41) at ($(base4) + (p1)$) {$1$};
        \node[vertex] (42) at ($(base4) + (p2)$) {$2$};
        \node[vertex] (44) at ($(base4) + (p5)$) {$5$};
        \draw[edge] (43) --(42) --(41);
        
        \node[draw,line width=1pt] at($(base4) + 0.5*(p1) + 0.5*(p5) + (label)$) {$P_5[3,1,2,5]$};

        \node[vertex] (53) at ($(base5) + (p3)$) {$3$};
        \node[vertex] (51) at ($(base5) + (p1)$) {$1$};
        \node[vertex] (52) at ($(base5) + (p2)$) {$2$};
        \node[vertex] (54) at ($(base5) + (p4)$) {$4$};
        \node[vertex] (55) at ($(base5) + (p5)$) {$5$};
        \draw[edge] (55) -- (54) -- (53) --(52) --(51);
        
        \node[draw,line width=1pt] at($(base5) + (label)$) {$P_5[3,1,2,5,4]$};
    \end{tikzpicture}

    \caption{A depiction of the induced subgraphs $P_5[\pi_1,\ldots,\pi_i]$ for the ordering $\pi=(3,1,2,5,4)$ of the path graph $P_5$.  Note that $\pi$ is an even sequence since each of these induced subgraphs have an even number of edges.  We also observe that $\pi^{-1}=(2,3,1,5,4)$ is an alternating permutation; see \Cref{prop:path}.} 
    \label{fig:evenEg}
\end{figure}

For example, \Cref{fig:evenEg} depicts an even sequence for the path graph on 5 vertices with vertex set $[5]$.  While not immediate, once can verify that even sequences for the path graph $P_n$ with vertex set $[n]$ are exactly inverses of alternating permutations of size $n$ (see \Cref{prop:path}); so $\nu(P_n)=\eta(P_n)$ in this case. Our main result shows that this equality holds for a substantially larger class of graphs.  

To state this result, we remind the reader that a graph is \textit{complete multipartite} if one can partition its vertices into sets $V_1,\ldots,V_r$ such that $u$ and $v$ are adjacent if and only if $u\in V_i,v\in V_j$ for some $i\ne j$.  We say that a graph is a \textit{blowup of a cycle} if one can partition its vertices into sets $V_1,\ldots,V_r$ such that $u$ and $v$ are adjacent if and only if $u\in V_i$ and $v\in V_{i+1}$ for some $i$ (with the indices written mod $r$).  With these definitions in mind, we have the following.
\begin{theorem}\label{thm:nuEta}
    If $G$ is a graph which is either bipartite, complete multipartite, or a blowup of a cycle, then $\nu(G)=\eta(G)$.
\end{theorem}

While the above shows $\nu(G)=\eta(G)$ for a large class of graphs $G$, equality does not hold in general.  In fact, the following result shows that the family of graphs from \Cref{thm:nuEta} is essentially the largest hereditary family of graphs for which equality holds.
\begin{theorem}\label{thm:induced}
    If $G$ is a connected graph such that $\nu(G')=\eta(G')$ for all induced subgraphs $G'\sub G$, then $G$ is either bipartite, complete multipartite, or a blowup of a cycle.
\end{theorem}

Our proof of \Cref{thm:induced} relies on a structural graph theory result which may be of independent interest; see \Cref{def:pans} and \Cref{prop:structure} for a precise statement. 

While we do not have a full understanding of $\nu(G)$ for arbitrary graphs, we are able to prove several other results regarding $\nu(G)$, such as the general bound $\nu(G)\le \eta(G)$. 
 We refer the reader to \Cref{prop:nuGeneral}  and \Cref{cor:nuGeneral} for a full list of these results.

\subsubsection{Extremal Values of $\nu(G)$ and $\eta(G)$}
We next turn to the extremal problem of studying the largest and smallest possible values of $\nu(G)$ and $\eta(G)$. For arbitrary $n$-vertex graphs this is an uninteresting problem, since  $\nu(\overline{K_n})=\eta(\overline{K_n})=n!$ and $\nu(K_n)=\eta(K_n)=0$ for $n\ge 2$  are easily seen to achieve the maximum and minimum possible values.  However, this problem becomes non-trivial when one looks at smaller classes of graphs.  To this end, we consider these extremal problems for trees.

To state our result, we recall that a tree is a \textit{star} $K_{1,n}$ if there is a single-non leaf vertex.  We say that a tree is a  \textit{hairbrush} if it consists of a path $v_0v_1\cdots v_n$ such that each vertex $v_i$ with $i\ge 1$ is adjacent to a leaf $u_i$.  That is, hairbrushes are ``comb graphs'' with an extra vertex $v_0$ attached at the end; see \Cref{fig: the hairbrush c2}.
\begin{figure}[h]
    \centering
    \begin{subfigure}[a]{.4\textwidth}
    \centering
    \begin{tikzpicture}
    \node[vertex] (v0) at (0,0) {$v_0$};
    \node[vertex] (v1) at (1,0) {$v_1$};
    \node[vertex] (v2) at (2,0) {$v_2$};
    \node[vertex] (v3) at (3,0) {$v_3$};
    \node[vertex] (u1) at (1,-1) {$u_1$};
    \node[vertex] (u2) at (2,-1) {$u_2$};
    \node[vertex] (u3) at (3,-1) {$u_3$};
    \draw[edge] (v0)--(v1)--(v2)--(v3);
    \draw[edge] (v1)--(u1);
    \draw[edge] (v2)--(u2);
    \draw[edge] (v3)--(u3);
    \end{tikzpicture}
    \caption{The hairbrush $H_3$}
    \label{fig: the hairbrush c2}
\end{subfigure}
\begin{subfigure}[a]{.4\textwidth}
    \centering
    \begin{tikzpicture}
    \node[vertex] (v0) at (0,1) {$v_0$};
    \node[vertex] (v1) at (-2.5,0) {$v_1$};
    \node[vertex] (v2) at (-1.5,0) {$v_2$};
    \node[vertex] (v3) at (-0.5,0) {$v_3$};
    \node[vertex] (v4) at (0.5,0) {$v_4$};
    \node[vertex] (v5) at (1.5,0) {$v_5$};
    \node[vertex] (v6) at (2.5,0) {$v_6$};
    \draw[edge] (v0)--(v1);
    \draw[edge] (v0)--(v2);
    \draw[edge] (v0)--(v3);
    \draw[edge] (v0)--(v4);
    \draw[edge] (v0)--(v5);
    \draw[edge] (v0)--(v6);
    \end{tikzpicture}
    \caption{The Star $K_{1,6}$}
    \label{fig: the star}
\end{subfigure}

    \label{fig:extremal trees}
\end{figure}

\begin{theorem}\label{thm:etaTree}
    If $T$ is a tree on $2n+1$ vertices, then
    \[n! 2^n\le \nu(T)=\eta(T)\le (2n)!\]
    Moreover, equality holds in the lower bound if and only if $T$ is a hairbrush, and equality holds in the upper bound if and only if $T$ is a star.
\end{theorem}

Note that the equality $\nu(T)=\eta(T)$ follows from \Cref{thm:nuEta}, and that $\nu(T)=\eta(T)=0$ if $T$ has an even number of vertices (since $e(T)$ is odd), which is why \Cref{thm:etaTree} only considers trees with an odd number of vertices.


\subsubsection{Multiplicity of Roots}
Lastly, we consider the problem of determining the multiplicity of $-1$ as a root of $A_D(t)$, and we denote this quantity by $\mult(A_D(t),-1)$.  Note that studying $\mult(A_D(t),-1)$ is related to studying $\nu(G)=|A_D(-1)|$ in the sense that a graph $G$ has $\nu(G)=0$ if and only if $\mult(A_D(t),-1)>0$ for every orientation $D$ of $G$

One of the first questions one might ask in this setting is to determine how large $\mult(A_D(t),-1)$ can be amongst $n$-vertex digraphs.  Trivially, $\mult(A_D(t),-1)\le e(D)$ (since the degree of $A_D(t)$ is at most $e(D)$), which implies $\mult(A_D(t),-1)\le {n\choose 2}$ if $D$ has $n$ vertices.  We prove a substantially stronger  upper bound which turns out to be sharp.
\begin{theorem}\label{thm root multiplicity upper bound}
    If $D$ is an $n$-vertex digraph, then \[\mult(A_D(t),-1)\le n-s_2(n),\] where $s_2(n)$ denotes the number of 1's in the binary expansion of $n$.  Moreover, for all $n$, there exist $n$-vertex digraphs $D$ with $A_D(t)=\frac{n!}{2^{n-s_2(n)}}(1+t)^{n-s_2(n)}$.
\end{theorem}

We also obtain a general lower bound on $\mult(A_D(t),-1)$.
\begin{prop}\label{prop:matching}
    Let $D$ be an orientation of an $n$-vertex graph $G$.  If every matching in the complement of $G$ has size at most $m$, then  $\mult(\A_D(t),-1) \geq \floor{\frac{n}{2}} - m$.
\end{prop}

Roughly speaking, \Cref{prop:matching} says that if $G$ is ``dense'' (i.e.\ if the complement of $G$  contains small only matchings), then $\mult(A_D(t),-1)$ will be large. While \Cref{prop:matching} is not tight in general, it turns out to be tight if $D$ is an orientation of the complete graph.

\begin{theorem}\label{thm: -1 multiplicity for tournaments}
    If $D$ is a tournament on $n$ vertices, then $\mult(A_D(t),-1)=\lfloor\frac{n}{2}\rfloor$.
\end{theorem}
More generally, we suspect that \Cref{prop:matching} is tight for orientations of complete multipartite graphs, see \Cref{conj:multipartite} for more.

Given \Cref{thm: -1 multiplicity for tournaments} and the fact that $|A_D(-1)|=|A_{D'}(-1)|$ whenever $D,D'$ are orientations of the same graph, it is perhaps natural to guess that $\mult(A_D(t),-1)$ depends only on the underlying graph of $D$.  This turns out to be false; see \Cref{fig:-1 multiplicity changes with orientation} for a counterexample.

    \begin{figure}[hbpt]
        \centering
        \begin{tikzpicture}

        \coordinate (a) at (0:2);
        \coordinate (b) at (72:2);
        \coordinate (c) at (144:2);
        \coordinate (d) at (216:2);
        \coordinate (e) at (288:2);

        \coordinate (base1) at (-6,0);
        \coordinate (base2) at (1,0);

        \node[vertex] (A1) at ($(a) + (base1)$) {$v_1$};
        \node[vertex] (B1) at ($(b) + (base1)$) {$v_2$};
        \node[vertex] (C1) at ($(c) + (base1)$) {$v_3$};
        \node[vertex] (D1) at ($(d) + (base1)$) {$v_4$};
        \node[vertex] (E1) at ($(e) + (base1)$) {$v_5$};

        \node[vertex] (A2) at ($(a) + (base2)$) {$v_1$};
        \node[vertex] (B2) at ($(b) + (base2)$) {$v_2$};
        \node[vertex] (C2) at ($(c) + (base2)$) {$v_3$};
        \node[vertex] (D2) at ($(d) + (base2)$) {$v_4$};
        \node[vertex] (E2) at ($(e) + (base2)$) {$v_5$};
        
        \draw[arrow] (A1) -- (B1);
        \draw[arrow,line width=2pt] (B1)-- (C1);
        \draw[arrow] (C1)-- (D1);
        \draw[arrow](D1) -- (E1);
        \draw[arrow] (A1) -- (C1);
        \draw[arrow] (C1) -- (E1);
        \draw[arrow] (B1) -- (D1);

        \draw[arrow] (A2) -- (B2);
        \draw[arrow,line width=2pt] (C2) -- (B2);
        \draw[arrow] (C2) -- (D2);
        \draw[arrow] (D2) -- (E2);
        \draw[arrow] (A2) -- (C2);
        \draw[arrow] (C2) -- (E2);
        \draw[arrow] (B2) -- (D2);
        
        \end{tikzpicture}
        \caption{Two orientations of the same graph with different $-1$ multiplicities. 
 The digraph on the left has $\A_{D_1}(t) = (1 + t)^3 (1 + t + 11 t^2 + t^3 + t^4)$ while the one on the right has $\A_{D_2}(t) = (1 + t) (1 + 5 t + 16 t^2 + 16 t^3 + 16 t^4 + 5 t^5 + t^6)$.}
        \label{fig:-1 multiplicity changes with orientation}
    \end{figure}
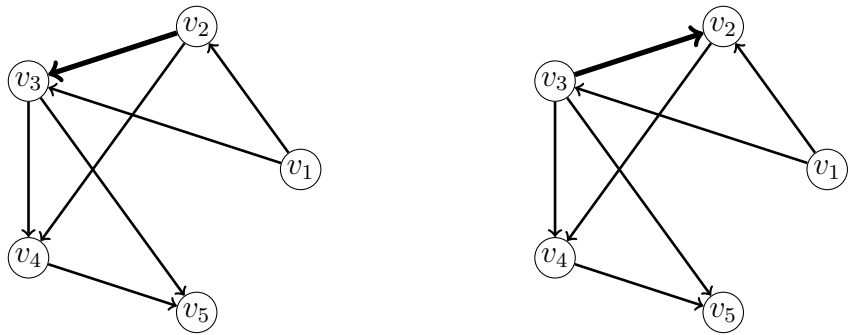

\subsection{Related Works}

Before presenting our proofs, we first comment on a variety of similar looking polynomials that appear in the literature.  This summary is by no means exhaustive, as there are countless objects adopting the monikers of \textit{Eulerian numbers}, \textit{Eulerian polynomials}, \textit{descents} or \textit{inversions}; most of which have little to no relation to the problems studied here.

As previously mentioned, Foata and Zeilberger \cite{FOATA199679} were the first to define the above polynomial $A_D(t)$. They were primarily interested in finding digraphs $D$ for which the statistic $\des_D$ has the same distribution on $\mathfrak{S}_D$ as another statistic $\maj_D$ which is a generalization of the classical ``major index" permutation statistic. Their problem has a positive answer when $D$ is ``bipartitional'', which in our terminology is a join of bidirected complete graphs and digraphs with no edges. In our \Cref{prop basic facts}(d) we consider more general joins of digraphs which recovers some of the results from \cite{FOATA199679}.

When $D$ is acyclic, $\des_D$ is a \textit{weighted-inversion statistic} as in Kadell \cite{KADELL198522} and Degengardt--Milne \cite{DEGENHARDT200049} which is any function $w:\mathfrak{S}_n\to \Z_{\geq 0}$ of the form
\[w(\sigma)=\sum_{\substack{1\leq i<j\leq n\\ \sig(i)>\sig(j)}} w_{i,j}\]
for some upper triangular matrix $(w_{i,j})$. In fact, it is easy to see that $\des_D$ encapsulates all weighted inversion statistics $w$ for which $w_{i,j}\in \{0,1\}$ for all $i,j$. 


Archer et al \cite{ARCHER2020112041} define a \textit{Eulerian polynomial for a family $\mathfrak{B}_n$ of digraphs} as
    \[b_n(u)=\sum_{D\in \mathfrak{B}_n}u^{\des(D)}\]
where $\des(D)$ is defined by \textit{fixing} a labeling $\omega_D:V(D)\to[n]$ of the vertices of $D$  for each $D$ and then counting the number of edges that go from larger label to smaller label.  The polynomial $\A_D(t)$ can be recovered from $b_n(u)$ by selecting $\mathfrak{B}_n$ to be \textit{all} labelings of a fixed digraph, although this particular choice of family of digraphs is not considered in \cite{ARCHER2020112041}.
The polynomial $A_D(t)$ also appears as specializations of the Ellzey \textit{chromatic quasisymmetric function for digraphs} \cite{Ellzey_2017} as well as the Awan and Bernardi \textit{B-polynomial} \cite{awan_bernardi_2020}.

Lastly, we note that Kalai \cite{KALAI2002412} explicitly mentioned the determination of $\A_D(-1)$ as an interesting problem in the context of the Condocet paradox in social choice theory, and Even-Zohar\cite{Even-Zohar2017} gave some basic properties of $\nu(G)$ in their study of the random framed knots through permutation statistics.

\subsection{Organization of the Paper}
The rest of this paper is organized as follows. In \Cref{sec not and basic facts}, we lay out the necessary definitions, notation and elementary properties of $A_D(t)$ for the rest of the paper. In \Cref{sec -1 evaluation} we consider evaluations of $\A_D(t)$ at $t = -1$. In particular, we give  basic properties of $\nu(G)$ in \Cref{subsec:basic}, we prove \Cref{thm:nuEta} in \Cref{subsec:interpretation}, and  we prove \Cref{thm:induced} in \Cref{subsec:induced}. In \Cref{sec bounds for nu on tree} we study the bounds for $\nu(T)$ for trees $T$ and prove \Cref{thm:etaTree}. In \Cref{section: multiplicity} we study the multiplicity of $-1$ as a root of $A_D(t)$ and prove \Cref{prop:matching} and Theorems \ref{thm root multiplicity upper bound} and \ref{thm: -1 multiplicity for tournaments}. We conclude the paper in \Cref{sec conclusion} with a few remarks and open problems regarding $A_D(t)$.

\section{Preliminaries}\label{sec not and basic facts}
\subsection{Notation}\label{sec preliminaries}

    Graphs in this paper will always be finite and simple. An \textit{oriented graph} is a digraph $D$ obtained by taking a graph $G$ and giving an orientation to each of its edges. In this case we say $G$ is the \textit{underlying graph} of $D$ and that $D$ is an \textit{orientation} of $G$.

    We will often denote the vertex set of a graph or digraph by $V(G)$ or $V(D)$ respectively, or simply $V$ whenever $G$ or $D$ is understood; and we similarly use the notation $E(G)$ and $E(D)$. For a subset $S\subseteq V(D)$, we write $D[S]$ for the induced subgraph of $D$ on $S$, and we write $D - S$ for the induced subgraph $D[V\setminus S]$.  We write $\overline{S}$ for the complement $V(D)\sm S$.  For two sets $S,T\subseteq V(D)$, we write $e_D(S,T)$ for the number of arcs  $uv$ whose tail $u$ is in $S$ and whose head $v$ is in $T$.

An \textit{integer composition} $\alpha$ of $n\in \N$ of length $\ell$ is a sequence $\alpha=(\alpha_1,\alpha_2,\dots,\alpha_\ell)$ of positive integers such that $\alpha_1+\cdots+\alpha_\ell=n$, and for such a sequence we write $\alpha\vdash n$. The elements $\alpha_1,\dots,\alpha_\ell$ of $\alpha$ are sometimes called the \textit{parts} of $\alpha$. If $\alpha_1 \geq \alpha_2 \geq \dots \geq \alpha_\ell$, we say that $\alpha$ is an \textit{integer partition}. The \textit{type} of an integer composition $\alpha$, denoted by $\lambda(\alpha)$, is the integer partition given by sorting the parts of $\alpha$ in weakly decreasing order. We will use the notation $(n)^m$ to denote the integer partition $({n,\dots,n})$ that has $m$ copies of $n$.  

An \textit{ordered set partition} of a set~$S$ is a sequence $P=(B_1,\dots,B_\ell)$ of mutually disjoint subsets $B_i$ of $S$ called \textit{blocks} such that $\bigcup_{i=1}^\ell B_i= S$. The \textit{type} of a set partition $P$ is the type of the integer composition $(|B_1|,\dots,|B_\ell)|$. An unordered set partition is an ordered set partition with the order forgotten. 

For positive integers $n_1,\dots,n_r$ and $n=n_1+\cdots+n_r$, the \textit{$t$-multinomial coefficient} is defined to be  \[\begin{bmatrix}
    n\\ n_1,\ldots,n_r
\end{bmatrix}_t=\frac{[n]_t!}{[n_1]_t!\cdots [n_r]_t!}\] where $[n]_t!=(1+t)(1+t+t^2)\cdots (1+t+t^2+\cdots+t^{n-1})$.


	

\subsection{Basic Properties}\label{sec basic facts}

In this subsection, we prove a number of basic facts regarding $\A_D(t)$ which will be used throughout the paper. We begin by establishing a list of elementary properties for $\A_D(t)$.  Some of these properties can be found in Even-Zohar \cite{Even-Zohar2017}; we have included their proofs for completeness. For this, we recall that a polynomial $f(t)=\sum_{k} a(k) t^k$ is \textit{palindromic with center $d/2$} if $a(k)=a(d-k)$ for all $k$. 

\begin{prop}\label{prop basic facts}
    Let $D=(V,E)$ be a directed graph with $n$ vertices and $m$ arcs.
     \begin{enumerate}
        \item[(a)] The polynomial $\A_D(t)$ is palindromic with center $m/2$.
        \item[(b)] If $(u,v),(v,u)\in E$, then $\A_D(t)=t\cdot \A_{D-(u,v)-(v,u)}(t)$.
        \item[(c)] If $D=\bigsqcup_{i=1}^r D_i$ is a disjoint union of digraphs of orders $n_1,\ldots,n_r$, then
        \[\A_D(t)=\binom{n}{n_1,\ldots,n_r} \cdot \prod_{i=1}^r A_{D_i}(t).\]
        \item[(d)] If $D$ is formed by taking the disjoint union of digraphs $D_1,\ldots,D_r$ of orders $n_1,\ldots,n_r$ and then adding all arcs of the form $u\to v$ with $u\in V_i,\ v\in V_j$ and with $i<j$, then
        \[\A_D(t)=\begin{bmatrix} n\\ n_1,\ldots,n_r\end{bmatrix}_t \cdot \prod_{i=1}^r A_{D_i}(t).\]
    \end{enumerate}
\end{prop}
    We note that (b) and (c) allows us to reduce our problems to studying digraphs $D$ which are orientations of connected graphs. 
\begin{proof}
    For (a), consider the map $\phi:\mathfrak{S}_D\to \mathfrak{S}_D$ which sends $\sigma\in \mathfrak{S}_D$ to $\tau\in \mathfrak{S}_D$ with $\tau(u)=n-\sigma(u)+1$ for all $u\in V$.  It is not difficult to see that this is an involution and that $\tau$ has $k$ descents if and only if $\sigma$ has $m-k$ descents.  From this it follows that $\A_D(t)$ is palindromic with center $m/2$.

    For (b), we observe that for every permutation of $D$, exactly one of the arcs $(u,v)$ and $(v,u)$ will be a $D$-descent, giving the result.

    For (c), let $D$, $D_i$ and $n_i$ be as in the statement of the proposition. A permutation $\sigma\in \mathfrak{S}_D$ can be made by choosing an ordered set partition $\pi=(B_1,\dots,B_r)$ of $[n]$ of type $(n_1,\dots,n_r)$ and then choosing a bijection $\sigma_i:V(D_i)\to B_i$ for each $i\in [r]$. For each $i$, we can view $\sigma_i$ as an element of $\mathfrak{S}_{D_i}$. Since $D$ is a disjoint union of digraphs, we have
    \[\des_D(\sigma)=\des_{D_1}(\sigma_1)+\cdots +\des_{D_r}(\sigma_r).\]
    Since there are $\binom{n}{n_1,\dots,n_r}$ ordered set partitions of type $(n_1,\dots,n_r)$, we have
    \begin{align*}
        A_D(t)=\sum_{\sigma\in \mathfrak{S}_D}t^{\des_D(\sigma)}
        =\sum_{(B_1,\dots,B_r)} \prod_{i=1}^r\lp\sum_{\sigma_i\in \mathfrak{S}_{D_i}}t^{\des_{D_i}(\sigma_i)}\rp
        =\binom{n}{n_1,\dots,n_r}\prod_{i=1}^r A_{D_i}
    \end{align*}


    
    For (d), we can again view $\sigma\in \mathfrak{S}_D$ as a tuple $(\pi,\sigma_1,\dots,\sigma_r)$ of ordered set partition $\pi$ of type $(n_1,\dots,n_r)$ and permutations $\sigma_i\in \mathfrak{S}_{D_i}$. Let $\mathfrak{M}_{n_1,\dots,n_r}$ be the set of words $w$ with $n_1$ 1's, $n_2$ 2's, etc. Ordered set partitions $\pi$ of type $(n_1,\dots,n_r)$ are in natural bijection with words $w\in \mathfrak{M}_{n_1,\dots,n_r}$ by the map that sends $\pi$ to the word $w$ whose $i$-th letter is $j$ if $i\in B_j$. Then, a $D$-descent $(u,v)$ either has $u,v\in V_i$ or $u\in V_i, v\in V_j$ with $i<j$: the former are counted by $\des_{D_j}(\sigma_j)$, and the latter are counted by the pairs $(w_{\sigma(u)},w_{\sigma(v)})$ with $w$ the word in bijection with $\pi$. Hence, we have
    \[\des_D(\sigma)=\inv(w)+\des_{D_1}(\sigma_1)+\cdots +\des_{D_r}(\sigma_r).\]
    Then we have
    \begin{align*}
        A_D(t)=\sum_{\sigma\in \mathfrak{S}_D}t^{\des_D(\sigma)}
        =\sum_{w\in \mathfrak{M}_{n_1,\dots,n_r}}t^{\inv(w)} \prod_{i=1}^r\lp\sum_{\sigma_i\in \mathfrak{S}_{D_i}}t^{\des_{D_i}(\sigma_i)}\rp
        =\sum_{w\in \mathfrak{M}_{n_1,\dots,n_r}}t^{\inv(w)} \prod_{i=1}^r A_{D_i}
    \end{align*}
    The result then follows from the well-known result (see \cite[Proposition 1.7.1]{stanley_2011}) that
    \[\sum_{w\in \mathfrak{M}_{n_1,\dots,n_r}}t^{\inv(w)}=\begin{bmatrix}
        n\\ n_1,\dots,n_r
    \end{bmatrix}_t.\qedhere\]
\end{proof}

Next, we have a lemma which allows us to express the Eulerian polynomial of a digraph $D$ in terms of Eulerian polynomials of induced subgraphs of $D$. 
 \begin{lemma}\label{lemma split into subgraphs}
     If $D=(V,E)$ is an $n$-vertex digraph and $k\in [n]$, then
     \[
         \A_D(t) = \sum_{S\in \binom{V}{k}} \frac{t^{e_D(S,\overline{S})} + t^{e_D(\overline{S},S)}}{2}\A_{D[S]}(t)\A_{D - S}(t).
     \]
 \end{lemma}
 \begin{proof} 
 
 For ease of notation we assume $V=[n]$. Fix $\sigma\in \mathfrak{S}_{D}$ and let $S\subseteq V(D)$ be the set such that $\sigma(S)$ is the interval $[k]=\{1,\dots,k\}$. We observe that if $u\to v$ is a descent for $\sigma$ in $D$, then one of the following must hold:
 \begin{itemize}
     \item Both $u$ and $v$ are in $S$
     \item Both $u$ and $v$ are in $\overline{S}$
     \item $v$ is in $S$ and $u$ is in $\overline{S}$ (since $\sig(S)=[k]$)
 \end{itemize} Therefore, if we set $\tau=\sigma|_{S}$ and $\tau'=\sigma|_{\overline{S}}$, we have
 \[
    \des_D(\sigma) = e_D(\overline{S},S)+\des_{D[S]}(\tau) + \des_{D - S}(\sigma') 
 \] Hence we have
     \begin{align*}
         \A_D(t) &=\sum_{\sigma\in \mathfrak{S}_n}t^{\des_D(\sigma)}\\
         &= \sum_{S\in \binom{[n]}{k} }\sum_{\substack{\sigma\in \mathfrak{S}_n\\ \sigma(S)=[k]}}t^{\des_D(\sigma)}\\
         &= \sum_{S\in \binom{[n]}{k}} \sum_{\substack{\tau\in \mathfrak{S}_S\\ \tau'\in \mathfrak{S}_{[n] - S} }}t^{e_D(\overline{S},S)+\des_{D[S]}(\tau) +\des_{D - S}(\tau') }\\
         &= \sum_{S\in \binom{[n]}{k}} t^{e_D(\overline{S},S)}\A_{D[S]}(t) \A_{D - S}(t)
     \end{align*}
     If we repeat this same argument but consider $\sigma$ and $S$ with $\sigma(S)=\{n - k+1,\dots,n\}$, then we get
     \[
        \A_D(t) = \sum_{S\in \binom{[n]}{k}} t^{e_D(S,\overline{S})}\A_{D[S]}(t) \A_{D - S}(t)
     \] By adding these two expressions for $\A_D(t)$ and dividing by 2, we find that
     \[
        \A_D(t) = \sum_{S\in \binom{[n]}{k}} \frac{t^{e_D(S,\overline{S})} + t^{e_D(\overline{S},S)}}{2}\A_{D[S]}(t)\A_{D - S}(t).\qedhere
     \]
 \end{proof}

Finally, we consider a construction which will be useful for \Cref{thm root multiplicity upper bound} and \Cref{prop:nuGeneral}. Given digraphs $D_1,D_2$ and a root vertex $v\in D_2$, the  \textit{rooted product digraph}, denoted $D_1\circ_v D_2$, is obtained by gluing a copy of $D_2$ at $v$ to each vertex of $D_1$, see \Cref{fig:rooted product} for an example.
\begin{figure}[h]
    \centering
    \begin{tikzpicture}
        \tikzset{shorten >= 6pt};
        \coordinate (P3base) at (-6,0);
        \coordinate (K3base) at (-2,0);
        \coordinate (prodbase) at (3,0);
        \coordinate (label) at (0,-3);

        \coordinate (p1) at (-2.25,0);
        \coordinate (p2) at (-0.75,0);
        \coordinate (p3) at (0.75,0);
        \coordinate (p4) at (2.25,0);

        \coordinate (k1) at (0,0);
        \coordinate (k2) at (1,-1);
        \coordinate (k3) at (0,-2);

         \draw[arrow] ($(P3base) + (p1)$) -- ($(P3base) + (p2)$);
         \draw[arrow] ($(P3base) + (p2)$) -- ($(P3base) + (p3)$);
         \draw[arrow] ($(P3base) + (p3)$) -- ($(P3base) + (p4)$);

        \draw[arrow] ($(prodbase) + (p1)$) -- ($(prodbase) + (p2)$);
        \draw[arrow] ($(prodbase) + (p2)$) -- ($(prodbase) + (p3)$);
        \draw[arrow] ($(prodbase) + (p3)$) -- ($(prodbase) + (p4)$);

        \foreach \v in {p1,p2,p3,p4}{
            \node[vertexSmall] at ($(P3base) + (\v)$) {};

            \draw[arrow] ($(prodbase) + (\v)$) -- ($(prodbase) + (\v) + (k2)$);
            \draw[arrow] ($(prodbase) + (\v) + (k2)$) -- ($(prodbase) + (\v) + (k3)$);
            \draw[arrow] ($(prodbase) + (\v) + (k3)$) -- ($(prodbase) + (\v)$);
            
            \node[vertexSmall,fill=black] at ($(prodbase) + (\v)$) {};
            \node[vertexSmall] at ($(prodbase) + (\v) + (k2)$) {};
            \node[vertexSmall] at ($(prodbase) + (\v) + (k3)$) {};
        }

       \draw[arrow] ($(K3base) + (k1)$) -- ($(K3base) + (k2)$);
       \draw[arrow] ($(K3base) + (k2)$) -- ($(K3base) + (k3)$);
       \draw[arrow] ($(K3base) + (k3)$) -- ($(K3base) + (k1)$);

        \node[vertexSmall,fill=black] at ($(K3base) + (k1)$) {};
        \node[vertexSmall] at ($(K3base) + (k2)$) {};
        \node[vertexSmall] at ($(K3base)  + (k3)$) {};

        \node[draw] at ($(P3base) + (label)$) {$\ora{P_4}$};
        \node[draw] at ($(K3base) + (label)$) {$\ora{K_3}$};
        \node[draw] at ($(prodbase) + (label)$) {$\ora{P_4}\circ_v \ora{K_3}$};

        \node[above right=2pt and 2pt of K3base]  {$v$};

    \end{tikzpicture}
    \caption{The rooted product digraph $\ora{P_4} \circ_v \ora{K_3}$ with the vertex $v$ highlighted in black.}
    \label{fig:rooted product}
\end{figure}
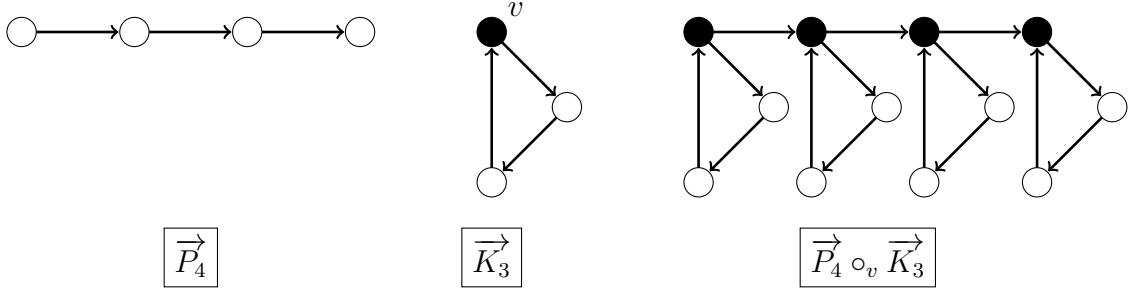

 This product was first defined by Godsil and McKay \cite{godsil1978new}, and it turns out that this operation plays very nicely with the Eulerian polynomial.
    \begin{prop}\label{prop: rooted product graphs}
        Let $D_1$ and $D_2$ be two digraphs on $m$ and $n$ vertices respectively. If $v\in D_2$, then 
        \[A_{D_1\circ_v D_2}(t)=\frac{1}{m!}\binom{mn}{n,\ldots,n}\cdot  A_{D_1}(t)A_{D_2}(t)^m.\]
        In particular, the polynomial is the same for any choice of root $v\in D_2$.
    \end{prop}
\begin{proof}
            To create a permutation $\sigma\in \mathfrak{S}_D$, we can do the following
            \begin{enumerate}
                \item Select a vector of permutations $(\sigma_1,\dots,\sigma_m)\in \prod_{i=1}^m\mathfrak{S}_{D_2}$.
                \item Select an ordered set partition $\pi=(B_1,\dots,B_m)$ of $[mn]$ of type $(n)^m$. For each $i$, if $B_i=\{b_1<b_2,\dots<b_m\}$, we think of $\sigma_i$ as a map $\sigma:V(D_2)\to B_i$ via $\sigma_i(v_k)=b_j$ if $\sigma(v_k)=j$.
                \item Select a permutation $\tau\in \mathfrak{S}_{D_1}$.
                \item For each $u\in D_1$, if $\tau(u)=i$, assign the permutation $\sigma_j$ with the $i$-th largest value at the root $v$ to the copy of $D_2$ at $u$.
            \end{enumerate}
            Let $T$ be the set of tuples $(\sigma_1,\dots,\sigma_{m};\pi;\tau)$ of permutations $\sigma_i$ of $D_2$, ordered partitions $\pi$ of type $(n)^m$, and permutations $\tau$ of $D_1$. Then, the above defines a map
            \[\phi:T\to \mathfrak{S}_{D}.\]
            For a given $\rho\in \mathfrak{S}_m$, we observe that two elements
            \[(\sigma_1,\dots,\sigma_{m};B_1,\dots,B_m;\tau)\quad\text{\quad and \quad}(\sigma_{\rho(1)},\dots,\sigma_{\rho(m)};B_{\rho(1)},\dots,B_{\rho(m)};\tau)\]
            of $T$ produce exactly the same element of $\mathfrak{S}_D$ under $\phi$. Hence, $\phi$ is an $m!$-to-1 map. 

            For $\psi\in \mathfrak{S}_D$, let $(\sigma_1\dots,\sigma_{m};\pi;\tau)$ be an element of the preimage $\phi^{-1}(\psi)$. The $D$-descents of $\psi$ come from either an edge in one of the copies of $D_2$ or an edge in $D_1$. The former are exactly the $D_2$-descents of the $\sigma_i$ (thinking of them as elements of $\mathfrak{S}_{D_2})$. For the latter, $(u,u')$ is a $D$-descent of $\psi$ between two vertices of $D_1$ if and only if the edge is a $D_1$ descent of $\tau$ because $\psi(u)=\sigma_j(v)$ if and only if $\sigma_j(v)$ is the $\tau(u)$-th largest element of $\{\sigma_1(v),\dots,\sigma_j(v)\}$ (thinking of $\sigma_i$ as a map $\sigma_i:V(D_2)\to B_i$). Hence, we have
            \[\des_D(\psi)=\des_{D_2}(\sigma_1)+\cdots+\des_{D_2}(\sigma_m)+\des_{D_1}(\tau).\]
            Since there are $\binom{mn}{n,\ldots,n}$ ordered set partitions of $[mn]$ of type $(n)^m$ and $\phi$ is $m!$-to-1, we have

            \begin{align*}
        A_{D}(t)&=\sum_{\psi\in \mathfrak{S}_D}t^{\des_D(\psi)}\\
        &=\frac{1}{m!}\sum_{(\sigma_1,\dots,\sigma_m;\pi;\tau)\in T}t^{\des_{D_2}(\sigma_1)+\cdots+\des_{D_2}(\sigma_m)+\des_{D_1}(\tau)}  \\
        &=\frac{1}{m!}\binom{nm}{n,\ldots,n}\lp \sum_{\sigma\in \mathfrak{S}_{D_2}}t^{\des_{D_2}(\sigma)}\rp^m\sum_{\tau\in \mathfrak{S}_{D_1}}t^{\des_{D_1}(\tau)}\\
        &=\frac{1}{m!}\binom{nm}{n,\ldots,n}A_{D_2}(t)^m A_{D_1}(t).\qedhere
    \end{align*}
        \end{proof}

\section{Combinatorial Interpretations of $\nu(G)$}\label{sec -1 evaluation}
In this section, we prove our results regarding $\nu(G):=|\A_D(-1)|$ where $D$ is any orientation of the graph $G$.  As noted in the introduction, $\nu(G)$ is intimately related to the quantity $\eta(G)$, which is the number of even sequences of $G$, i.e.\ the number of orderings $\pi$ of $V(G)$ such that the induced subgraphs $G[\pi_1,\ldots,\pi_i]$ have an even number of edges for all $i$.  As a warmup, we formally establish the connection between alternating permutations and even sequences of paths mentioned in the introduction.
\begin{prop}\label{prop:path}
    Let $P_n$ denote the path graph with vertex set $[n]$.  A permutation $\pi$ of $[n]$ is an even sequence of $P_n$ if and only if $\pi^{-1}$ is an alternating permutation.
\end{prop}
\begin{proof}
    First assume $\pi^{-1}$ is an alternating permutation (which in particular means $n$ is odd), and define $G_j:=P_n[\pi_1,\ldots,\pi_j]$.
    \begin{claim}
        Fix $1\le i\le n$ and let $j=\pi^{-1}_i$.  If $i$ is even then $e(G_j)-e(G_{j-1})=2$, and if $i$ is odd then $e(G_j)-e(G_{j-1})=0$.
    \end{claim}
    \begin{proof}
        First assume $i$ is even (which in particular means $1<i<n$).  Since $\pi^{-1}$ is an alternating permutation, $j:=\pi^{-1}_i>\pi^{-1}_{i-1},\pi^{-1}_{i+1}$.  This means both of $i$'s neighbors in $P_n$ (namely $i-1$ and $i+1$) lie in $\{\pi_1,\ldots,\pi_{j-1}\}$, so $e(G_j)-e(G_{j-1})=2$ as claimed.  

        Now assume $i$ is odd.  Because $\pi^{-1}$ is an alternating permutation, $\pi^{-1}_i$ is less than all of the neighbors of $i$ in $P_n$, so $e(G_j)-e(G_{j-1})=0$ as desired.
    \end{proof}
    Because $e(G_j)-e(G_{j-1})$ is even for all $j$, and since $e(G_1)=0$ is even, we conclude that $e(G_j)$ is even for all $j$, proving that $\pi$ is an even sequence.

    Now assume $\pi$ is an even sequence, i.e.\ that the induced subgraphs $G_j:=P_n[\pi_1,\ldots,\pi_j]$ have an even number of edges for all $j$.  In particular, $n$ must be odd in order for $G_n=P_n$ to have an even number of edges.
    
    \begin{claim}
        For all $1\le i<n$, if $i$ is odd then $\pi_i^{-1}<\pi_{i+1}^{-1}$, and otherwise $\pi_i^{-1}>\pi_{i+1}^{-1}$.
    \end{claim}
    \begin{proof}
        We prove the result by induction on $i$ starting with the base case $i=1$.  Assume for contradiction that $\pi_2^{-1}<\pi_1^{-1}:=j$.  This implies $e(G_j)-e(G_{j-1})=1$ (since $\pi_j=1$ has exactly one neighbor amongst the set $\{\pi_1,\ldots,\pi_{j-1}\}$, namely $2$), contradicting $e(G_j),e(G_{j-1})$ both being even.  Thus we must have $\pi_1^{-1}<\pi_2^{-1}$.

        Inductively assume we have proven the result up to some value $i>1$.  If $i$ is odd and $\pi_{i+1}^{-1}<\pi_i^{-1}:=j$, then $e(G_j)-e(G_{j-1})=1$ (since  $\pi_j=i$ has a unique neighbor in $\{\pi_1,\ldots,\pi_{j-1}\}$, namely $\pi_{i+1}^{-1}$ due to the inductive hypothesis $\pi_{i-1}^{-1}>\pi_i^{-1}$), a contradiction.  If $i$ is even and $\pi_{i+1}^{-1}>\pi_i^{-1}:=j$, then again $e(G_j)-e(G_{j-1})=1$ (since  $\pi_j=i$ has a unique neighbor in $\{\pi_1,\ldots,\pi_{j-1}\}$, namely $\pi_{i-1}^{-1}$ due to the inductive hypothesis $\pi_{i-1}^{-1}<\pi_i^{-1}$).  With this we conclude the claim.
    \end{proof}
    This claim, together with the observation that $n$ must be odd, shows that $\pi^{-1}$ is an alternating permutation, completing the proof.
\end{proof}


\subsection{Basic Properties of $\nu(G)$}\label{subsec:basic}
In this subsection we prove several basic properties of $\nu(G)$, some of which will be needed for the proof of \Cref{thm:nuEta}.   We begin with a basic but important observation.
\begin{lemma}\label{prop odd edges to evaluation at -1}
    If $G$ has an odd number of edges, then $\nu(G)=0$.
\end{lemma}
\begin{proof}
    Let $m=e(G)$ and $D$ be any orientation of $G$.  By Proposition~\ref{prop basic facts}(a), $\A_D(t) $ is palindromic with center $(m-1)/2$.  Since $(-1)^k = -(-1)^{m - k}$ for $m$ odd, it follows that $\A_D(-1)=0$.
\end{proof}

The remainder of our proofs for this section rely heavily on the following special case of \Cref{lemma split into subgraphs}.

 \begin{cor}\label{lemma remove a vertex}
     If $D$ is an $n$-vertex digraph, then
     \[
        \A_{D}(t) = \sum_{v\in V} \frac{t^{\deg_D^+(v)} + t^{\deg_D^-(v)}}{2}\A_{D - v}(t)
     \]
 \end{cor}
 \begin{proof}
     Applying \Cref{lemma split into subgraphs} for $k = 1$ gives
     \[
         \A_D(t) = \sum_{S\in \binom{V}{1}} \frac{t^{e_D(S,\overline{S})} + t^{e_D(\overline{S},S)}}{2}\A_{D[S]}(t)\A_{D - S}(t).
     \]
     Note that for $S=\{v\}$, we have $e_D(S,\overline{S}) = \deg_D^+(v)$,  $e_D(\overline{S},S) = \deg_D^-(v)$, and $A_{D[S]}(t)=1$, giving the result.
 \end{proof}

For ease of notation, we will often write the summation symbol $\sum_{v\in V}$ in \Cref{lemma remove a vertex} simply as $\sum_v$ or even $\sum$.  This result immediately gives the following. 
\begin{cor}\label{cor:nuInequality}
    We have $\nu(G)\le \sum_v \nu(G-v)$.
\end{cor}
\begin{proof}
    When $t=-1$, each term in the summation of \Cref{lemma remove a vertex} has coefficients in $\{-1,0,1\}$.  Taking absolute values on both sides and using the triangle inequality gives for any orientation $D$ of $G$ that
    \[\nu(G)=|A_D(-1)|\le \sum_v |A_{D-v}(-1)|=\sum_v \nu(G-v).\]
\end{proof}

An analog of the result above holds for even sequences.
\begin{lemma}\label{lem:etaRecurrence}
    If $G$ is a graph with an odd number of edges, then $\eta(G)=0$.  Otherwise $\eta(G)=\sum_v \eta(G-v)$.
\end{lemma}
\begin{proof}
    If $e(G)$ is odd then there exist no even sequences (since $e(G[\pi_1,\ldots,\pi_n])=e(G)$ is always odd), so $\eta(G)=0$.  Assume now that $e(G)$ is even and let $\eta_v(G)$ denote the number of even sequences of $G$ with $v_n=v$.  Then  $\eta(G)=\sum_v \eta_v(G)$, and it is not difficult to see $\eta_v(G)=\eta(G-v)$ (since $e(G[\pi_1,\ldots,\pi_{n-1},v])=e(G)$ is even for any permutation $\pi$ of $V(G)-v$). This gives the result.
\end{proof}

Finally, we introduce two graph operations that play nicely with $\nu(G)$.  Given a set of graphs $G_1,\ldots,G_r$ on disjoint vertex sets, the \textit{join} $\bigvee_{i=1}^r G_i$ is the graph obtained by taking the disjoint union of the $G_i$ graphs and then adding all possible edges between each of the $G_i$ graphs.  Given graphs $G_1,G_2$ and a root vertex $v\in G_2$, the \textit{rooted product graph} $G_1\circ_v G_2$ is obtained by gluing a copy of $G_2$ at $v$ to each vertex of $G_1$.  With all this established, we can state the following results involving $\nu(G)$.

\begin{prop}\label{prop:nuGeneral}
    Let $G$ be an $n$-vertex graph.
    \begin{enumerate}
        \item[(a)] We have $\nu(G)\le \eta(G)$.
        \item[(b)] We have $\nu(G)\le \sum_{v} \nu(G-v)$.
        \item[(c)] If $G=\bigsqcup_{i=1}^r G_i$ is a disjoint union of graphs of orders $n_1,\ldots,n_r$, then \[\nu(G)=\binom{n}{n_1,\ldots,n_r} \cdot \prod_{i=1}^r \nu(G_i).\]
        \item[(d)] If $G=\bigvee_{i=1}^r G_i$ is the join of graphs of orders $n_1,\ldots,n_r$, then
        \[\nu(G)=\left|\begin{bmatrix} n\\ n_1,\ldots,n_r\end{bmatrix}_{-1} \right|\cdot \prod_{i=1}^r \nu(G_i).\]
        \item[(e)] If $G=G_1\circ_v G_2$ with $|V(G_i)|=n_i$, then
        \[\nu(G)=\frac{1}{n_1!}{n_1n_2\choose n_2,\ldots,n_2} \nu(G_1)\nu(G_2)^{n_1}.\]
    \end{enumerate}
\end{prop}

\begin{proof} 
Note that (b) is \Cref{cor:nuInequality}, (c) and (d) follow from \Cref{prop basic facts}, and (e) follows from \Cref{prop: rooted product graphs}.  It thus remains to prove (a), which we do by induction on $V(G)$.

The base case is trivial, so assume we have proven the result up to some order $n$ and let $G$ be a graph of order $n$.  If $G$ has an odd number of edges then $\nu(G)=\eta(G)=0$ by Lemmas~\ref{prop odd edges to evaluation at -1} and \ref{lem:etaRecurrence}.  Otherwise by \Cref{cor:nuInequality},
    \[\nu(G)\le \sum_{v} \nu(G-v)\le \sum_{v} \eta(G-v)=\eta(G),\]
    where the second inequality used the inductive hypothesis and the equality used \Cref{lem:etaRecurrence}.
\end{proof}
\Cref{prop:nuGeneral} has a number of nice consequences.  For example, (c) and (d) imply that to determine $\nu(G)$ for all graphs $G$, it suffices to do this for graphs $G$ such that $G$ and its complement are both connected.  We also have the following immediate consequences.
\begin{cor}\label{cor:nuGeneral}
    Let $G$ be an $n$-vertex graph.
    \begin{enumerate}
        \item[(a)] If $G$ has a component with an odd number of edges, then $\nu(G)=0$.
        \item[(b)] If every vertex of $G$ has odd degree, then $\nu(G)=0$.
        \item[(c)] If $G$ has a vertex $v$ of degree $n-1$,  then $\nu(G)=0$ if $n$ is even, and otherwise $\nu(G)=\nu(G-v)$.
    \end{enumerate}
\end{cor}
\begin{proof}
    Part (a) follows from \Cref{prop:nuGeneral}(c) and the fact that $\nu(G)=0$ whenever $G$ has an odd number of edges by \Cref{prop odd edges to evaluation at -1} (or alternatively by \Cref{prop:nuGeneral}(a)).  

    For (b), we observe that for any ordering $\pi$ of $V(G)$, either the graph $G=G[\pi_1,\ldots,\pi_n]$ or the graph $G[\pi_1,\ldots,\pi_{n-1}]$ has an odd number of edges.  Thus $\eta(G)=0$, and hence $\nu(G)=0$ by \Cref{prop:nuGeneral}(a).  

    For (c), we observe that $G$ is the join of $K_1$ together with $G-v$, so by \Cref{prop:nuGeneral}(d) we have $\nu(G)=|[n]_{-1} |\cdot 1\cdot \nu(G-v)$, and this equals 0 if $n$ is even and otherwise equals $\nu(G-v)$ as desired.
\end{proof}

\subsection{Proof of \Cref{thm:nuEta}}\label{subsec:interpretation}
In this section, we prove \Cref{thm:nuEta}, which we recall says that if $G$ is a graph that is either bipartite, complete multipartite, or a blowup of a cycle, then $\nu(G)=\eta(G)$.

The proofs for each of these cases follows the same basic strategy: We first show that for some ``natural'' orientation $D$ of $G$, we can easily predict the sign of $\A_D(-1)$.  From this we deduce $\nu(G)=\sum \nu(G-v)$, and hence that $\nu(G)=\eta(G)$ since the statistics $\nu,\eta$ satisfy the same recurrence relation.

We begin with the following ``natural'' orientations for bipartite graphs.

\begin{lemma}\label{lem:bipartiteOrientation}
    Let $D$ be a digraph such that one can partition its vertex set into $U\cup V$ such that every arc $u\to v$ of $D$ has $u\in U$ and $v\in V$. Then
    \[\A_D(-1)\ge 0,\]
    and if $D$ has an even number of arcs, then
    \[\A_D(-1)=\sum_{v\in V(D)}\A_{D-v}(-1).\]
\end{lemma}
\begin{proof}
    We first establish the equality for $D$ with an even number of arcs.   By Corollary~\ref{lemma remove a vertex}, we have
    \[\A_D(-1)=\sum_{u\in U}\frac{(-1)^{\deg^+(u)}+1}{2}\A_{D-u}(-1)+\sum_{v\in V}\frac{1+(-1)^{\deg^-(v)}}{2}\A_{D-v}(-1).\]
    We claim that for $u\in U$, we have $\frac{(-1)^{\deg^+(u)}+1}{2}\A_{D-u}(-1)=\A_{D-u}(-1)$.  Indeed, this is immediate if $\deg^+(u)$ is even.  If $\deg^+(u)$ is odd, then since $D$ has an even number of arcs, $D-u$ has an odd number of arcs.  By \Cref{prop odd edges to evaluation at -1}, $\A_{D-u}(-1)=0$, so again the claim trivially holds.  An analogous result holds for the $v$ terms, and applying these claims to the inlined equation above gives the result. 

    We next prove $\A_D(-1)\ge 0$ by using induction on $|V(D)|$, the base case being trivial.  If $D$ has an odd number of arcs then this quantity is 0 by \Cref{prop odd edges to evaluation at -1}, so we may assume $D$ has an even number of arcs.  Thus, by the result proven above, we have
    \[\A_D(-1)=\sum_{v\in V(D)}\A_{D-v}(-1)\ge 0,\]
    with the last step using the inductive hypothesis on each of the digraphs $D-v$ (each of which continues to satisfy the hypothesis of the lemma).  This completes the proof.
\end{proof}
\begin{cor}\label{cor:nuBipartite}
    If $G$ is a bipartite graph with an odd number of edges, then $\nu(G)=0$, and otherwise $\nu(G)=\sum_v \nu(G-v)$.
\end{cor}
\begin{proof}
    The result when $G$ has an odd number of edges follows from \Cref{prop odd edges to evaluation at -1}, so assume $G$ has an even number of edges, and let $D$ be an orientation of $G$ as in the previous lemma.  Having $\A_D(-1)\ge 0$ implies
    \[\nu(G)=\A_D(-1)=\sum_v \A_{D-v}(-1)=\sum_v\nu(G-v),\]
    where the second equality used the second part of \Cref{lem:bipartiteOrientation} and the last equality used $\nu(G-v)=\A_{D-v}(-1)$ since this latter quantity is non-negative by \Cref{lem:bipartiteOrientation}.
\end{proof}

We next turn to orientations of complete multipartite graphs.  We begin by establishing the following simple criteria for determining if $\nu(G)=0$.

\begin{lemma}\label{lem:multipartiteOdd}
    Let $G$ be a complete multipartite graph on $V_1\cup \cdots \cup V_r$. If $|V_i|$ is odd for at least two values of $i$, then $\nu(G)=\eta(G)=0$.
\end{lemma}
\begin{proof}
    We will show in this case that $\eta(G)=0$, i.e.\ that there exist no even sequences for $G$.  From this it will follow from $\nu(G)\le \eta(G)$ of \Cref{prop:nuGeneral}(a) that $\nu(G)=0$ as well.

    Assume for contradiction that $\pi$ is an even sequence of $G$.  Let $j$ be the smallest integer such that $|V_i\cap \{\pi_1,\ldots,\pi_j\}|$ is odd for at least two values of $i$, noting that such a (smallest) integer exists since this holds for $j=n$ by hypothesis.  Since $j$ is the smallest integer with this property, there must be exactly two integers $i$ such that $|V_i \cap \{\pi_1,\ldots,\pi_j\}|$ is odd, say this holds for $i=a,b$.  Since $G$ is complete multipartite, the number of edges of $G[\pi_1,\ldots,\pi_j]$ is exactly
    \[\sum_{i<i'} |V_i\cap \{\pi_1,\ldots,\pi_j\}|\cdot |V_{i'}\cap \{\pi_1,\ldots,\pi_j\}|.\]
    Exactly one term in this sum is odd, namely the one with $\{i,i'\}=\{a,b\}$.  This implies $G[\pi_1,\ldots,\pi_j]$ has an odd number of edges, contradicting $\pi$ being an even sequence.
\end{proof}
We next turn to the ``natural'' orientation of complete multipartite graphs.
\begin{lemma}\label{lem:multipartiteOrientation}
    Let $D$ be a digraph with vertex set $V_1\cup \cdots \cup V_r$ and arcs $u\to v$ if and only if $u\in V_i,\ v\in V_j$ and $i<j$.  Then \[\A_D(-1)\ge 0,\]
    and if $|V_i|$ is odd for at most one value of $i$, then
    \[\A_D(-1)=\sum_v \A_{D-v}(-1).\]
\end{lemma}
\begin{proof}
    As in the bipartite case, we begin by establishing the equality.  Suppose at most one of the parts of $D$ has odd size. By \Cref{lemma remove a vertex} we have
    \[\A_D(-1)=\sum_{v\in V(D)} \frac{(-1)^{\deg^+(v)}+(-1)^{\deg^-(v)}}{2} \A_{D-v}(-1),\]
    so it suffices to show that for each $v\in V(D)$, either $\deg^+(v),\deg^-(v)$ are both even or $\A_{D-v}(-1)=0$.  Suppose $v\in V_i$. Then $\deg^+(v)=|\bigcup_{j>i} V_j|$ and $\deg^-(v)=|\bigcup_{j<i} V_j|$.  If $|V_{i'}|$ is even for all $i'\ne i$, then $\deg^+(v)$ and $\deg^-(v)$ will be even.
    If $|V_{i'}|$ is odd for some $i'\ne i$, then $|V_i|$ must be even by hypothesis, so $|V_i-v|$ is odd.  This means $D-v$ is the orientation of a complete multipartite graph with two parts of odd size, namely $V_i-v$ and $V_{i'}$.  By the previous lemma this implies $\A_{D-v}(-1)=0$, completing the proof of this part.

    The proof that $\A_D(-1)\ge 0$ follows essentially the same inductive proof as in \Cref{lem:bipartiteOrientation}.  We omit the details.
\end{proof}
From these lemmas, the proof of \Cref{cor:nuBipartite} carries over to give the following.
\begin{cor}\label{cor:nuMultipartite}
    If $G$ is a complete multipartite graph with at least two parts of odd size, then $\nu(G)=0$, and otherwise $\nu(G)=\sum_v \nu(G-v)$.
\end{cor}

Finally, we prove our lemmas for graphs $G$ which are blowups of cycles, which we recall means that one can partition the vertex set of $G$ into sets $V_1,\ldots,V_r$ (which we will call the \textit{parts} of $G$) such that $uv$ is an edge of $G$ if and only if $u\in V_i$ and $v\in V_{i+1}$ for some $i$, where the indices are taken modulo $r$.  Again we begin with a simple criteria for having $\nu(G)=0$.

\begin{lemma}\label{lem:cycleOdd}
    Let $G$ be a blowup of a cycle with parts $V_1,\ldots,V_r$. If $|V_i| |V_{i+1}|$ is odd for an odd number of integers $1\le i\le r$, then $\nu(G)=\eta(G)=0$.
\end{lemma}
\begin{proof}
    By the definition of $G$ being a blowup of a cycle, we have $e(G)=\sum_{i=1}^r |V_i||V_{i+1}|$.  Thus if $|V_i|V_{i+1}|$ is odd for an odd number of integers, then $e(G)$ is odd.  This implies $\eta(G)=0$ and hence $\nu(G)=0$ by \cref{prop:nuGeneral}(a).
\end{proof}

Our analog of Lemmas~\ref{lem:bipartiteOrientation} and \ref{lem:multipartiteOrientation} will be slightly more complex in the setting of blowups of cycles.  For this, we define our ``natural'' directed analog of blowups of cycles as follows: we say a digraph $D$ is a \textit{blowup of a directed $r$-cycle} if it has vertex set $V_1\cup \cdots \cup V_r$ and arcs of the form $u\to v$ if and only if $u\in V_i$ and $v\in V_{i+1}$ for some $i$.  For such a digraph, we define $m(D)$ to be the number of integers $1\le i\le r$ such that $|V_i|$ and $|V_{i+1}|$ are both odd, i.e.\ such that $|V_i||V_{i+1}|$ is odd. 

\begin{lemma}\label{lem:cycleOrientation}
    Let $D$ be a blowup of an $r$-cycle.  
    \begin{enumerate}
        \item[(a)] If $m(D)$ is odd, then $\A_D(-1)=0$.
        \item[(b)] If $m(D)$ is even, then 
        \[(-1)^{m(D)/2} A_D(-1)\ge 0\]
        and
        \[(-1)^{m(D)/2}\A_D(-1)=\sum_{v} (-1)^{m(D-v)/2} \A_{D-v}(-1)\]
    \end{enumerate}
\end{lemma}
Note that the first statement implies that when $m(D-v)$ is odd, $\A_{D-v}(-1)=0$. Hence, the sum in the second statement is a well-defined real number. 
\begin{proof}
    For (a), note that if $m(D)$ is odd, then $D$ is an orientation of a graph $G$ as in \Cref{lem:cycleOdd}, so $A_D(-1)=0$ as desired.  It thus remains to prove (b), and we begin by establishing the sum. 
    
    By \Cref{lemma remove a vertex}, we have
    \[\A_D(-1)=\sum_{v\in V(D)} \frac{(-1)^{\deg^+(v)}+(-1)^{\deg^-(v)}}{2} \A_{D-v}(-1),\]
    so to prove the desired sum, it suffices to show that for each $v\in V(D)$, either $\A_{D-v}(-1)=0$, or $\deg^+(v),\deg^-(v)$ both have the same parity as $(m(D-v)-m(D))/2$.  Note that by (a) we have $\A_{D-v}(-1)=0$ if $m(D-v)$ is odd, so from now on we may assume $m(D-v)$ is even (and hence it makes sense to talk about the parity of $(m(D-v)-m(D))/2$ since $m(D)$ is assumed to be even).
    
    Suppose $v\in V_i$, which means $\deg^+(v)=|V_{i+1}|$ and $\deg^-(v)=|V_{i-1}|$.  If $|V_{i-1}|\not\equiv_2 |V_{i+1}|$, then $m(D-v)=m(D)+1$ if $|V_i|$ is even and $m(D-v)=m(D)-1$ if $|V_i|$ is odd.  Since $m(D)$ is even, $m(D-v)$ is odd in either case, which we assumed not to be the case.  Thus we must have $|V_{i-1}|\equiv_2 |V_{i+1}|$.  
    
    If both $|V_{i-1}|$ and $|V_{i+1}|$ are even, then $m(D-v)=m(D)$, and hence $\deg^+(v)=|V_{i+1}|,\deg^-(v)=|V_{i-1}|$ have the same parity as $(m(D-v)-m(D))/2$.  If instead both these quantities are odd, then $m(D-v)=m(D)+2$ if $|V_i|$ is even and $m(D-v)=m(D)-2$ if $|V_i|$ is odd. Hence, we have $m(D-v)-m(D)=\pm 2$ and so \[(m(D-v)-m(D))/2=\pm 1\equiv_2 |V_{i\pm 1}|=\deg^{\pm}(v),\] so again in this case the desired result follows.  This completes the proof of the equality.

    The proof that $(-1)^{m(D)/2}\A_D(-1)\ge 0$ again follows from essentially the same inductive proof as in \Cref{lem:bipartiteOrientation}.  More precisely, by the equality we just proved we find
    \[(-1)^{m(D)/2} \A_D(-1)=\sum_v (-1)^{(m(D-v)} \A_{D-v}(-1)\ge0,\]
    with this last inequality using the inductive hypothesis.
\end{proof}

Again these lemmas give the following corollary.
\begin{cor}
    If $G$ is a blowup of a cycle such that $|V_i| |V_{i+1}|$ is odd for an even number of integers $i$, then $\nu(G)=0$, and otherwise $\nu(G)=\sum_v \nu(G-v)$.
\end{cor}
\begin{proof}
    \Cref{lem:cycleOdd} implies the first half of the result.  Otherwise, if $D$ is the directed blowup of an $r$-cycle whose underlying graph is $G$, then $m(D)$ is even and \Cref{lem:cycleOrientation} inductively gives
    \[\nu(G)=(-1)^{m(D)/2} \A_D(-1)=\sum_v (-1)^{m(D-v)/2}\A_{D-v}(-1)=\sum_v \nu(G-v),\]
    completing the proof.
\end{proof}

We are now ready to prove our main result for this subsection.
\begin{proof}[Proof of \Cref{thm:nuEta}]
    We aim to show that $\nu(G)=\eta(G)$ whenever $G$ is bipartite, complete multipartite, or a blowup of a cycle.  We first consider the case that $G$ is bipartite.  We prove this result by induction on $|V(G)|$, the base case $\nu(K_1)=\eta(K_1)=1$ being trivial.    By \Cref{cor:nuBipartite} and \Cref{lem:etaRecurrence}, if $G$ has an odd number of edges then $\nu(G)=\eta(G)=0$, and otherwise
    \[\nu(G)=\sum_v \nu(G-v)=\sum_v \eta(G-v)=\eta(G),\]
    where the middle equality used the inductive hypothesis (and that $G-v$ is bipartite whenever $G$ is).  
    
    Nearly identical arguments work for the cases when $G$ is either complete multipartite or a blowup of a cycle, completing the proof.
\end{proof}

It is tempting to try to generalize the approach of this subsection by finding ``natural'' orientations of other graphs in order to show $\nu(G)=\sum \nu(G-v)$; see for example \Cref{conj:Eulerian}.  However, we emphasize that \Cref{thm:induced} shows that the inductive proof of \Cref{thm:nuEta} can not be extended beyond the class of graphs which are bipartite, complete multipartite, or blowups of cycles.

Before moving on, we note the following cute consequence of our results for $\nu(G)$ which gives a combinatorial interpretation for $t$-multinomial coefficients evaluated at $-1$.
\begin{cor}\label{cor:multinomial}
    If $n_1,\ldots,n_r$ are positive integers and $n=n_1+\cdots+n_r$, then
    \[
    \left|\begin{bmatrix}
   n\\ n_1,\ldots,n_r
\end{bmatrix}_{-1}\right|=\begin{cases}
0&\text{at least two parts of odd size}\\
    \binom{\lfloor n/2\rfloor}{\lfloor n_1/2\rfloor,\dots,\lfloor n_r/2\rfloor}&\text{otherwise.}
\end{cases}
\]

\end{cor}
It is likely that \Cref{cor:multinomial} is already well known in the literature, though the only concrete source we are aware of is \cite[Section 5.2]{ajose_2007} which solves the case $r=2$ (from which the general result can be derived).
\begin{proof}[Sketch of Proof]
    Let $G$ be the complete multipartite graph with parts of sizes $n_1,\ldots,n_r$.  Since $G$ is the join of independent sets of size $n_i$, \Cref{prop:nuGeneral}(d) implies $\nu(G)=\left|\begin{bmatrix}
    n_1+\cdots +n_r\\ n_1,\ldots,n_r
\end{bmatrix}_{-1}\right| \prod n_i!$.  On the other hand, by using ideas similar to those in \Cref{lem:multipartiteOdd}, one can work out that the number of even sequences $\eta(G)$ equals  $\prod n_i!$ times the number of words $w$ consisting of $n_1$ $1$'s, $n_2$ $2$'s, and so on, with the additional property that each prefix $w_1\cdots w_i$ has all but at most one letter appearing an even number of times.  This is equivalent to saying that $w_i=w_{i+1}$ for all odd $i<\sum n_j$, so the number of these words is 0 if $n_i$ is odd for at least two values of $i$, and otherwise equals $ \binom{\lfloor n/2\rfloor}{\lfloor n_1/2\rfloor,\dots,\lfloor n_r/2\rfloor}$.  By \Cref{thm:nuEta} we have $\nu(G)=\eta(G)$, giving the desired result.
\end{proof}

\subsection{Proof of \Cref{thm:induced}}\label{subsec:induced}
In this subsection we characterize which graphs have $\nu(G')=\eta(G')$ for all induced subgraphs $G'\sub G$.  For this the following will be crucial.
\begin{definition}\label{def:pans}
    We define the \textit{odd pan graph} $C_{2k+1}^*$ to be the graph obtained by taking the odd cycle $C_{2k+1}$ and then adding a new vertex $u$ adjacent to exactly one vertex of $C_{2k+1}$; see \Cref{fig: c7star}.  We say that a graph $G$ is \textit{odd pan-free} if it contains no induced subgraph which is isomorphic to $C_{2k+1}^*$ for any $k\ge 1$.
\end{definition}
We note that some authors use the term ``odd pan'' only to refer to $C_{2k+1}^*$ when $k\ge 2$, but we emphasize that we include the paw graph $C_3^*$ in our definition of odd pans. Our motivation for this definition is the following lemma.

\begin{lemma}\label{lem:pans}
    We have $\nu(C_{2k+1}^*)\ne \eta(C_{2k+1}^*)$ for all $k\ge 1$.
\end{lemma}

\begin{figure}[h]
    \centering
    \begin{subfigure}[b]{.4\textwidth}
    \begin{tikzpicture}
    \tikzstyle{vertex} = [circle,draw=black, minimum size=2pt];

        \coordinate (4) at (-90:2);
        \coordinate (4sub) at (-90:4);
        \coordinate (5) at ({-90+1*51.42}:2);
        \coordinate (6) at ({-90+2*51.42}:2);
        \coordinate (7) at ({-90+3*51.42}:2);
        \coordinate (1) at ({-90+4*51.42}:2);
        \coordinate (2) at ({-90+5*51.42}:2);
        \coordinate (3) at ({-90+6*51.42}:2);

        \node[vertex] (v4) at (4) {$v_4$};
        \node[vertex] (u) at (4sub) {$u$};
        \node[vertex] (v5) at (5) {$v_5$};
        \node[vertex] (v6) at (6) {$v_6$};
        \node[vertex] (v7) at (7) {$v_7$};
        \node[vertex] (v1) at (1) {$v_1$};
        \node[vertex] (v2) at (2) {$v_2$};
        \node[vertex] (v3) at (3) {$v_3$};

         \draw (u) -- (v4)--(v5)--(v6)--(v7)--(v1)--(v2)--(v3)--(v4);
         \node (sp) at (0,0) {~};
         \node (sp) at (3,0) {~};
         \node (sp) at (-3,0) {~};
        \end{tikzpicture}
        \caption{$C_7^*$}
        \label{fig: c7star}
    \end{subfigure}
    \begin{subfigure}[b]{.4\textwidth}
    
        \begin{tikzpicture}
\tikzstyle{vertex} = [circle,draw=black, minimum size=1pt];
        \coordinate (4) at (-90:2);
        \coordinate (4sub) at (-90:4);
        \coordinate (5) at ({-90+1*51.42}:2);
        \coordinate (6) at ({-90+2*51.42}:2);
        \coordinate (7) at ({-90+3*51.42}:2);
        \coordinate (1) at ({-90+4*51.42}:2);
        \coordinate (2) at ({-90+5*51.42}:2);
        \coordinate (3) at ({-90+6*51.42}:2);

        \node[vertex] (v4) at (4) {$5$};
        \node[vertex] (u) at (4sub) {$1$};
        \node[vertex] (v5) at (5) {$2$};
        \node[vertex] (v6) at (6) {$6$};
        \node[vertex] (v7) at (7) {$3$};
        \node[vertex] (v1) at (1) {$7$};
        \node[vertex] (v2) at (2) {$4$};
        \node[vertex] (v3) at (3) {$8$};

         \draw (u) -- (v4)--(v5)--(v6)--(v7)--(v1)--(v2)--(v3)--(v4);
         \node (sp) at (0,0) {~};
         \node (sp) at (3,0) {~};
         \node (sp) at (-3,0) {~};
        \end{tikzpicture}
        \caption{An even sequence}
        \label{fig: c7star even sequence}
    \end{subfigure}
    \caption{}
    \label{fig:c7star and c7star even sequence}
\end{figure}
\begin{proof}
    We prove this by showing $\nu(C_{2k+1}^*)=0$ and $ \eta(C_{2k+1}^*)>0$.  Let $v_1,\ldots,v_{2k+1}$ denote the vertices of the odd cycle of $C_{2k+1}^*$ and $u$ the pendant vertex, say with $u$ adjacent to $v_{k+1}$.  

    Define the sequence $(x_1,\ldots,x_{2k+2})$ by having $x_1=u$ and $x_i=v_{k+2i-2}$ for all $i\ge 1$, with these indices for $v$ written modulo $2k+1$; see \Cref{fig: c7star even sequence}.  The first $k+1$ elements \[\{x_1,\ldots,x_{k+1}\}=\{u,v_{k+2},v_{k+4},\ldots,v_{k-1}\}\] form an independent set, and hence $C^*_{2k+1}[x_1,\ldots,x_i]$ has no edges for all $1\le i\le k+1$. 
    For $i>k+1$, we have that 
    \[k+2i-2=k+2(i-k-1)+2(k+1)-2\equiv_{2k+1}k+2(i-k-1)-1.\]
    Therefore, $x_i=v_{k+2i-2}=v_{k+2(i-k-1)-1}$ is adjacent to $x_{i-k-1}=v_{k+2(i-k-1)-2}$ and $x_{i-k}=v_{k+2(i-k)-2}$. Thus $C_{2k+1}^*[x_1,\ldots,x_i]$ is even for all $i$, so $(x_1,\ldots,x_{2k+2})$ is an even sequence and $\eta(C_{2k+1}^*)>0$.

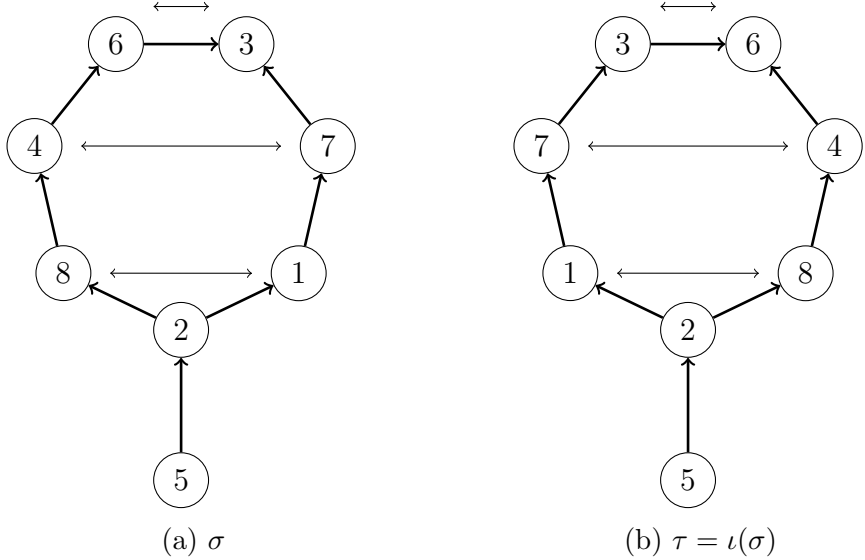
\begin{figure}[h]
    \centering
    \begin{subfigure}[b]{.4\textwidth}
        \begin{tikzpicture}
        \tikzstyle{vertex} = [circle,draw=black, minimum size=2pt];
        \tikzstyle{arrow} = [line width=1pt,->];

        \coordinate (4) at (-90:2);
        \coordinate (4sub) at (-90:4);
        \coordinate (5) at ({-90+1*51.42}:2);
        \coordinate (6) at ({-90+2*51.42}:2);
        \coordinate (7) at ({-90+3*51.42}:2);
        \coordinate (1) at ({-90+4*51.42}:2);
        \coordinate (2) at ({-90+5*51.42}:2);
        \coordinate (3) at ({-90+6*51.42}:2);

        \node[vertex] (v4) at (4) {$2$};
        \node[vertex] (u) at (4sub) {$5$};
        \node[vertex] (v5) at (5) {$1$};
        \node[vertex] (v6) at (6) {$7$};
        \node[vertex] (v7) at (7) {$3$};
        \node[vertex] (v1) at (1) {$6$};
        \node[vertex] (v2) at (2) {$4$};
        \node[vertex] (v3) at (3) {$8$};

         \draw[arrow] (u) -- (v4);
         \draw[arrow](v4)--(v5);
         \draw[arrow](v5)--(v6);
         \draw[arrow](v6)--(v7);
         \draw[arrow](v4)--(v3);
         \draw[arrow](v3)--(v2);
         \draw[arrow](v2)--(v1);
         \draw[arrow](v1)--(v7);
         \node (sp) at (0,0) {~};
         \node (sp) at (3,0) {~};
         \node (sp) at (-3,0) {~};
         \draw[<->,shorten >=.25cm ,shorten <=.25cm](v3)--(v5);
         \draw[<->,shorten >=.25cm ,shorten <=.25cm](v2)--(v6);
         \draw[<->,shorten >=.5cm ,shorten <=.5cm]($(v1)+(0,.5)$)--($(v7)+(0,.5)$);
        \end{tikzpicture}
        \caption{$\sigma$}
        \label{fig: c7star iota1}
    \end{subfigure}
    \begin{subfigure}[b]{.4\textwidth} 
        \begin{tikzpicture}
\tikzstyle{vertex} = [circle,draw=black, minimum size=2pt];
        \coordinate (4) at (-90:2);
        \coordinate (4sub) at (-90:4);
        \coordinate (5) at ({-90+1*51.42}:2);
        \coordinate (6) at ({-90+2*51.42}:2);
        \coordinate (7) at ({-90+3*51.42}:2);
        \coordinate (1) at ({-90+4*51.42}:2);
        \coordinate (2) at ({-90+5*51.42}:2);
        \coordinate (3) at ({-90+6*51.42}:2);

        \node[vertex] (v4) at (4) {$2$};
        \node[vertex] (u) at (4sub) {$5$};
        \node[vertex] (v5) at (5) {$8$};
        \node[vertex] (v6) at (6) {$4$};
        \node[vertex] (v7) at (7) {$6$};
        \node[vertex] (v1) at (1) {$3$};
        \node[vertex] (v2) at (2) {$7$};
        \node[vertex] (v3) at (3) {$1$};

         \draw[arrow] (u) -- (v4);
         \draw[arrow](v4)--(v5);
         \draw[arrow](v5)--(v6);
         \draw[arrow](v6)--(v7);
         \draw[arrow](v4)--(v3);
         \draw[arrow](v3)--(v2);
         \draw[arrow](v2)--(v1);
         \draw[arrow](v1)--(v7);
         \draw[<->,shorten >=.25cm ,shorten <=.25cm](v3)--(v5);
         \draw[<->,shorten >=.25cm ,shorten <=.25cm](v2)--(v6);
         \draw[<->,shorten >=.5cm ,shorten <=.5cm]($(v1)+(0,.5)$)--($(v7)+(0,.5)$);
         \node (sp) at (0,0) {~};
         \node (sp) at (3,0) {~};
         \node (sp) at (-3,0) {~};
    \end{tikzpicture}
        \caption{$\tau=\iota(\sigma)$}
        \label{fig: c7star iota2}
    \end{subfigure}
    \caption{The involution $\iota$}
    \label{fig:c7star involution}
\end{figure}
    
    To show $\nu(C_{2k+1}^*)=0$, let $D$ be an orientation of $C_{2k+1}^*$ such that $u\to v_{k+1},$ $v_1\to v_{2k+1}$, and for $0\le i\le k-1$, $v_{k+1+i}\to v_{k+1+(i+1)}$ and $v_{k+1-i}\to v_{k+1-(i+1)}$; see \Cref{fig: c7star iota1}. Define a map $\iota:\mathfrak{S}_V\to \mathfrak{S}_V$ sending $\sigma$ to $\tau:=\iota(\sigma)$ defined by setting $\tau(u)=\sigma(u)$ and 
    \[\tau(v_{k+1+i})=\begin{cases}
        \sigma(v_{k+1-i})&1\leq |i|\leq k\\
        \sigma(v_{k+1})&\text{otherwise}
    \end{cases}\]
    Then, $\iota$ is clearly an involution with no fixed points. By the orientation $D$ of the graph, we have that $(v_{k+1\pm i},v_{k+1\pm (i+1)})$ is a descent of $\tau$ if and only $(v_{k+1\mp i},v_{k+1\mp(i+1)})$ is a descent of $\sigma$ and that $(u,v_{k+1})$ is a descent of $\tau$ if and only if $(u,v_{k+1})$ is a descent of $\sigma$. Finally, we have that $(v_1,v_{2k+1})$ is a descent of $\tau$ if and only if  $(v_1,v_{2k+1})$ is not a descent of $\sigma$. Thus, $\iota$ changes the number of descents by exactly 1 and hence it is a sign-reversing involution, proving $A_D(-1)=0$.
    \end{proof}

Recall that \Cref{thm:induced} says that if $G$ is a connected graph, then $\nu(G')=\eta(G')$ for all induced subgraphs $G'\sub G$ if and only if $G$ is bipartite, complete multipartite, or a blowup of a cycle.  In view of the lemma above, it will suffice to prove the following structural graph theory lemma.
\begin{prop}\label{prop:structure}
    If $G$ is a connected graph, then $G$ is odd pan-free if and only if it is either bipartite, complete multipartite, or a blowup of a cycle.
\end{prop}
We will prove this proposition through the following two lemmas.

\begin{lemma}\label{lem:multipartitePan}
    If $G$ be a connected graph which is odd pan-free and which contains a triangle, then $G$ is a complete multipartite graph.
\end{lemma}

\begin{proof}
    Let $r$ be the maximum size of a clique of $G$, and note that $r\geq 3$ by hypothesis. Let $H$ be a maximal induced subgraph of $G$ which is isomorphic to a complete $r$-partite graph with non-empty parts, say with parts $V_1,\ldots,V_r$. Note that any $r$-clique of $G$ is trivially a complete $r$-partite induced subgraph, so such a maximal induced subgraph exists. We claim that $H=G$. 
    
    Suppose not, and let $u\in G\setminus H$. Since $G$ is connected, there is a path from $u$ to $H$, so we shall assume that $u$ is adjacent to $H$. Now, $u$ is not adjacent to some $v_i\in V_i$ for all $i$, as this would imply that $u$ together with the $v_i$ form an $(r+1)$-clique in $G$. Hence without loss of generality, we may assume $u$ is not adjacent to any vertex in $V_1$ and that it is adjacent to some $v_2\in V_2$.  If $u$ is not adjacent to some $v_k\in V_k$ for $k\geq 3$, then $u,v_1,v_2,v_k$ forms a copy of $C_3^*$ in $G$, which is a contradiction to $G$ being odd pan-free. Thus, $u$ is adjacent to every element of $V_3,\dots,V_r$, and critically we observe that $V_3\ne \emptyset$ since $r\ge 3$. 
    
    We claim that $u$ is adjacent to every element of $V_2$. Suppose there is some $v_2'\in V_2$ which is not adjacent to $u$.  Since $V_3$ is nonempty, we can take any $v_3\in V_3$ (which is adjacent to $u$) and form a $C_3^*$ out of $u,v_1,v_2',v_3$, which is a contradiction. Thus $u$ is adjacent to every element of $V_2$, as well as every element of $V_3,\ldots,V_r$, and is not adjacent to any element of $V_1$. This means $\{u\}\cup V_1,V_2,\dots,V_r$ forms an induced complete $r$-partite subgraph of $G$ that contains $H$, a contradiction to the maximality of $H$.  We conclude that $H=G$, completing the proof.  
    \end{proof}

The next lemma deals with the case when $G$ is triangle-free.  Here we recall that a graph is a blowup of a cycle if it has vertex set $V_1\cup \cdots \cup V_r$ and edges $uv$ if and only if $u\in V_i$ and $v\in V_{i+1}$ for some $i$. 
\begin{lemma}\label{lem:cyclePan}
    If $G$ be a connected graph which is odd pan-free and which is triangle-free but not bipartite, then $G$ is a blowup of a cycle.
\end{lemma}
\begin{proof}
    Assume that the shortest odd cycle of $G$ has length $2k+1$, noting that such a cycle exists with $2k+1\ge 5$ by hypothesis of $G$ being non-bipartite and triangle-free.  Let $H$ be a maximal induced subgraph of $G$ which is isomorphic to a blowup of a cycle of length $2k+1$, and let $V_1,\ldots,V_{2k+1}$ be its parts.  We claim that $H=G$.  

    Suppose not, and let $u\in G\setminus H$. Since $G$ is connected, there is a path from $u$ to $H$, so we shall assume that $u$ is adjacent to $H$, say that it is adjacent to $v_1\in V_1$. 

    We claim that $u$ is not adjacent to any vertex in $V_2\cup V_{2k+1}$.  Indeed, if $u$ was adjacent to some $v_2\in V_2$, then $u,v_1,v_2$ would form a triangle in $G$, a contradiction.  A symmetric argument shows $u$ can not be adjacent to any vertex in $V_{2k+1}$.

    We claim that $u$ is not adjacent to any vertex in $V_i$ for $4\le i\le 2k-1$.  Indeed, assume for contradiction that $u$ is adjacent to some $v_i\in V_i$, and for each $j\ne 1,i$ let $v_j$ be some vertex in $V_j$.  Observe that if $i$ is odd, then the vertices $v_1,u,v_i,v_{i+1},\ldots,v_{2k+1}$ form an odd cycle of length $2k+1-(i-3)$ (since it excludes vertices from $V_2,\ldots,V_{i-1}$ but includes $u$), contradicting $G$ having no odd cycles of length shorter than $2k+1$.  Similarly if $i$ is even then $v_1,u,v_i,v_{i-1},\ldots,v_2$ gives a contradiction.

    We claim that $u$ can not be adjacent to vertices in both $V_3$ and  $V_{2k}$.  Indeed, say it were adjacent to some $v_3\in V_3$ and $v_{2k}\in V_{2k}$ and let $v_i\in V_i$ for all other $i$.  Then $u,v_3,v_4,\ldots,v_{2k}$ is a cycle of length $2k-1$ in $G$, a contradiction.

    We claim that there exists some $i\ne 1$ such that $u$ is adjacent to every vertex of $V_i$.  Indeed, if for all $i\ne 1$ there existed a $v_i\in V_i$ which $u$ was not adjacent to, then $u,v_1,\ldots,v_{2k+1}$ would induce a $C_{2k+1}^*$ in $G$, a contradiction.

    With all of the claims above, we can assume $u$ is adjacent to some $v_1\in V_1$, every vertex of $v_3\in V_3$, and that it is adjacent to no vertices in $\bigcup_{i\ne 1,3} V_i$.  A symmetric argument to the previous claim shows that $u$ must be adjacent to every vertex of $V_1$.  Hence $V_1,V_2\cup \{u\},V_3,\ldots,V_{2k+1}$ induce a larger blowup of a cycle of length $2k+1$ in $G$, a contradiction.  We conclude that $H=G$ as desired.
\end{proof}

With these two lemmas we can easily prove \Cref{prop:structure}, and again we recall that this immediately implies \Cref{thm:induced} when combined with \Cref{lem:pans}.

\begin{proof}[Proof of \Cref{prop:structure}]
    It is straightforward to verify that complete multipartite graphs, blowups of cycles, and bipartite graphs are all odd pan-free (with this result also implicitly following from \Cref{thm:nuEta} and \Cref{lem:pans}).  If $G$ is a connected odd pan-free graph which contains a triangle, then \Cref{lem:multipartitePan} implies that $G$ is complete multipartite.  Otherwise $G$ is either bipartite or \Cref{lem:cyclePan} implies that $G$ is a blowup of a cycle, completing the proof.
\end{proof}

\section{Optimal bounds on \texorpdfstring{$\nu(T)$}{} for trees}\label{sec bounds for nu on tree}
Here we prove \Cref{thm:etaTree}, which we recall says that if $T$ is a tree on $2n+1$ vertices, then
\[n! 2^n\le \nu(T)=\eta(T)\le (2n)!,\]
with equality holding in the lower bound if and only if $T$ is a hairbrush, and equality holding in the upper bound if and only if $T$ is a star.

To aid with our proofs, given a tree $T$, we define 
\[\X(T)=\{x\in V(T):\textrm{each component of }T-x\textrm{ has an even number of edges}\},\]
and we will denote this simply by $\X$ whenever $T$ is understood.  Our motivation for this definition is the following.
\begin{lemma}\label{cor:treeInductiveHypothesis}
	If $T$ is a tree with an even number of edges, then
	\[\nu(T)=\sum_{x\in \X} \nu(T-x).\]
\end{lemma}
\begin{proof}
	By \Cref{cor:nuBipartite}, we have
	\begin{equation}\nu(T)=\sum_{x\in V(T)} \nu(T-x)=\sum_{x\in \X} \nu(T-x)+\sum_{x\in V(T)\setminus \X} \nu(T-x)=\sum_{x\in \X} \nu(T-x),\label{eq:X}\end{equation}
	where the last equality follows from \Cref{cor:nuGeneral}(a).
\end{proof}
With this lemma in mind, the idea for the proofs of the upper and lower bounds is as follows: we first apply \Cref{cor:treeInductiveHypothesis} and then  use induction to bound each of the terms $\nu(T-x)$ in the sum. Finally, we bound our total sum in terms of  $|\X|$ and show that equality can only occur when $|\X|=1$.

Throughout our proofs, we make heavy use of the fact that if $T'$ is a graph on $2n$ vertices with connected components $T_1,\ldots,T_r$ and $n_i=|V(T_i)|$, then
\begin{equation}\nu(T')={2n\choose n_1,\ldots,n_r} \prod_{i=1}^r \nu(T_i),\label{eq:treeDisjoint}\end{equation}
which follows from \Cref{prop:nuGeneral}(c).

\subsection{Lower bound for $\nu(T)$}
Here we prove that $\nu(T)\geq n!2^n$ for all trees $T$ with $2n+1$ vertices with equality when $T$ is the \textit{hairbrush $H_n$}. Recall that this graph $H_n$ is obtained by starting with a path $v_0-v_1-\cdots-v_n$ and then adding a leaf $u_i$ to each $v_i$ for $i\in [n]$; see \Cref{fig: the hairbrush c3}.  We begin by observing the following.

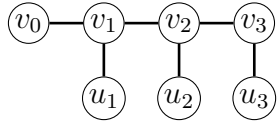
\begin{figure}[h]
	\centering
	\begin{tikzpicture}
		\node[vertex] (v0) at (0,0) {$v_0$};
		\node[vertex] (v1) at (1,0) {$v_1$};
		\node[vertex] (v2) at (2,0) {$v_2$};
		\node[vertex] (v3) at (3,0) {$v_3$};
		\node[vertex] (u1) at (1,-1) {$u_1$};
		\node[vertex] (u2) at (2,-1) {$u_2$};
		\node[vertex] (u3) at (3,-1) {$u_3$};
		\draw[edge] (v0)--(v1)--(v2)--(v3);
		\draw[edge] (v1)--(u1);
		\draw[edge] (v2)--(u2);
		\draw[edge] (v3)--(u3);
	\end{tikzpicture}
	\caption{The hairbrush $H_3$}
	\label{fig: the hairbrush c3}
\end{figure}

\begin{lemma}\label{lem:hairbrushEta}
	For $n\geq 0$, the hairbrush $H_n$ has $\nu(H_n)=n!2^n$. 
\end{lemma}
\begin{proof}
	Note that $\nu(H_0)=1=0! 2^0$, so from now on we assume $n>0$.  Since $\{v_0,u_1,\dots,u_n\}$ are leaves and $e(H_n)$ is even, none of these vertices are in $\X$. For $i\in [n-1]$, one of the components of $H_n-v_i$ is the subgraph on $\{v_{i+1},u_{i+1},\dots,v_n,u_{n}$\}, which has an odd number of edges, so $v_i\notin \X$. On the other hand, $H_n-v_n=H_{n-1}\sqcup \{u_n\}$, so $v_n\in \X$. Hence $\X=\{v_n\}$, so by \Cref{cor:treeInductiveHypothesis} and \eqref{eq:treeDisjoint} we have 
	\[\nu(H_n)=\nu(H_n-v_n)=\nu(H_{n-1}\sqcup \{u_n\}) =2n\cdot \nu(H_{n-1}).\]
	This provides a recurrence relation for $\nu(H_n)$ for $n\ge 1$, which combined with the initial condition $\nu(H_0)=1$ gives the desired formula.\qedhere
	
\end{proof}

In view of \Cref{cor:treeInductiveHypothesis}, to show that $\nu(T)>0$ (let alone that $\nu(T)\ge n! 2^n$), it is necessary to show the following.

\begin{lemma}\label{lem:treeNonzero}
	If $T$ is a tree with an even number of edges, then $\X\ne \emptyset$.
\end{lemma}
\begin{proof}
	We prove the result by induction on $n=|V(T)|$, the case $n=1$ being trivial. Suppose $n>1$ and let $x_1\cdots x_k$ be a longest path in $T$.  Note that every neighbor of $x_2$ other than $x_3$ is a leaf (as otherwise we could extend the path).  If $\deg(x_2)$ is even, then $T-x_2$ is the disjoint union of $\deg(x_2)-1$ copies of $K_1$ and a tree $T'$ with $e(T)-\deg(x_2)\equiv_2 0$ edges, so $x_2\in \X $.  Thus we may assume $\deg(x_2)$ is odd. 
	
	Let $T^*$ be $T$ after deleting all of the neighbors of $x_2$ other than $x_3$. Observe that $T^*$ is a tree with an even number of edges and with strictly fewer vertices than $T$ (since we have deleted $x_1$, in particular).  By the inductive hypothesis, there exists some vertex $y\in \X(T^*)$, i.e.\  $y$ is such that each connected component of $T^*-y$ has an even number of edges. Each component of $T^*-y$ is either a component of $T-y$ or it contains $x_2$. In the later case, the component of $T-y$ containing $x_2$ has $\deg(x_2)-1$ more edges than that in $T^*-y$. Since $\deg(x_2)-1$ is even by assumption, we have $y\in \X$, completing the proof.
\end{proof}

We will also need the following simple arithmetic inequality. 

\begin{lemma}\label{lem:treeArithmetic}
	Let $n$ be a non-negative integer and $(k_1,\ldots,k_r)$ a sequence of non-negative integers such that $r$ is even, and such that $n=r/2+\sum k_i$.  Then
	\[\prod_{i=1}^r \frac{k_i!}{2(2k_i+1)!}\ge \frac{(n-1)!}{2(2n-1)!},\]
	with equality if and only if $r=2$ and $\{k_1,k_2\}=\{n-1,0\}$.
\end{lemma}
\begin{proof}
	Let $\vec{k}=(k_1,\ldots,k_r)$ be a sequence as in the hypothesis of the lemma such that $\prod \frac{k_i!}{2(2k_i+1)!}$ is as small as possible. Without loss of generality, we may assume that $\vec{k}$ is weakly decreasing. We will prove the result by first showing that $\vec{k}$ is of the form $(k_1,0,0,\dots,0)$, and then that $r=2$.
	
	First assume for contradiction that $k_1\ge k_2>0$, and define the sequence $(k'_1,\ldots,k'_r)$ by $k'_i=k_i$ if $i>2$ and $k'_1=k_1+1,\ k'_2=k_2-1$.  Note that this sequence continues to satisfy the hypothesis of the lemma.  We claim that
	\[\prod_{i=1}^r \frac{k_i!}{2(2k_i+1)!}> \prod_{i=1}^r\frac{k_i'!}{2(2k_i'+1)!}.\]
	Since $k_i=k_i'$ for $i>2$, this is equivalent to saying
	\[\frac{k_1!k_2!}{4(2k_1+1)!(2k_2+1)!}>\frac{(k_1+1)!(k_2-1)!}{4(2k_1+3)!(2k_2-1)!},\]
	which further simplifies to
	\[\frac{k_2}{(2k_2+1)(2k_2)}> \frac{k_1+1}{(2k_1+3)(2k_1+2)} \iff\frac{1}{2k_2+1}> \frac{1}{2k_1+3}\]
	and this last bound holds since $k_1\ge k_2$.  This contradicts $\vec{k}$ being a minimizer, so we  conclude that $k_2=0$.
	
	Hence, we must have $\vec{k}=(k_1,0,\ldots,0)$, where necessarily $k_1=n-r/2$ by the hypothesis of the lemma.  In this case,
	\begin{equation*}\label{eq: function of n and r to make increasing}
		\prod_{i=1}^n \frac{k_i!}{2(2k_i+1)!}= \frac{k_1!}{2(2k_1+1)!}\cdot \frac{1}{2^{r-1}}=\frac{(n-r/2)!}{2^r(2n-r+1)!}.\end{equation*}
	Thus, to conclude the result it suffices to show that this function is strictly increasing for even $r\le n/2$, i.e.\ that for $r<n/2$
	\[\frac{(n-r/2)!}{2^r(2n-r+1)!}<\frac{(n-r/2-1)!}{2^{r+2}(2n-r-1)!}.\]
	This is equivalent to saying
	\[\frac{n-r/2}{(2n-r+1)(2n-r)}<\frac{1}{4},\]
	and indeed this quickly follows since $2n-r+1>2n-r\ge 2$.
\end{proof}
We now prove our lower bound for trees, which we restate below.
\begin{prop}\label{prop:lowerTree}
	If $T$ is a tree on $2n+1$ vertices, then $\nu(T)\geq 2^nn!$ with equality if and only if $T$ is the hairbrush $H_n$.
\end{prop}
\begin{proof}
	We prove the result by induction on $n$, the $n=0$ case being trivial.  Assume we have proven the result up to some value $n$ and let $T$ be a tree on $2n+1$ vertices.
	
	\begin{claim}
		For each $x\in \X$, we have $\nu(T-x)\ge n!2^n$ with equality only if $T-x$ is the disjoint union of $K_1$ and a hairbrush $H_{n-1}$.
	\end{claim}
	\begin{proof}
		let $T_1,\dots,T_r$ be the connected components of $T-x$, say with $n_i=|V(T_i)|$. Since each $T_i$ has an even number of edges, $n_i=2k_i+1$ for some non-negative integer $k_i$. By \eqref{eq:treeDisjoint} and induction, we have  \begin{equation}\nu(T-x)={2n\choose n_1,\ldots,n_r} \prod_{i=1}^r \nu(T_i)\ge {2n\choose n_1,\ldots,n_r} \prod_{i=1}^r k_i! 2^{k_i}.\label{eq:treeInductive}\end{equation} 
		Using $\sum_{i=1}^r k_i=\sum_{i=1}^r (n_i-1)/2=n-r/2$, we see that the quantity above can be rewritten as
		
		\begin{align}
			\binom{2n}{n_1,\dots,n_r}\prod_{i=1}^r k_i!2^{k_i}&=\frac{(2n)!}{n_1!\cdots n_r!}2^{k_1+\cdots+k_r}\prod_{i=1}^r k_i! \nonumber\\ \nonumber
			&=\frac{(2n)!}{(2k_1+1)!\cdots (2k_r+1)!}2^{n-r/2}\prod_{i=1}^r k_i!\\ \nonumber
			&=(2n)!2^{n-r/2}\prod_{i=1}^r \frac{k_i!}{(2k_i+1)!}\\ \nonumber
			&=(2n)!2^{n+r/2}\prod_{i=1}^r \frac{k_i!}{2(2k_i+1)!}\\ 
			&\geq (2n)!2^{n+r/2}\frac{(n-1)!}{2(2n-1)!} \label{eq:hairbrush}\\ 
			&= n! 2^{n+r/2-1}\geq n!2^{n},\nonumber
		\end{align}
		where \eqref{eq:hairbrush} used \Cref{lem:treeArithmetic}.  This proves the desired inequality of the claim.  Moreover, \Cref{lem:treeArithmetic} implies that equality in \eqref{eq:hairbrush} can only occur if $r=2$ and if, say, $T_1$ has one vertex and $T_2$ has $2n-1$ vertices.   Moreover, by induction we know \eqref{eq:treeInductive} can only hold with equality if $T_2$ is a hairbrush, proving the claim.
	\end{proof}
	By \Cref{cor:treeInductiveHypothesis} and the claim above, we have
	\begin{equation}\nu(T)=\sum_{x\in \X} \nu(T-x)\ge |\X|\cdot n! 2^n\ge n! 2^n,\label{eq:oneX}\end{equation}
	with this last inequality using \Cref{lem:treeNonzero}.  This gives the desired lower bound on $\nu(T)$. 
	
	Now suppose $\nu(T)=n! 2^n$.  This implies both inequalities of \eqref{eq:oneX} are equalities, which can only hold if $|\X|=1$, say with $\X=\{x\}$; and if $T-x$ consists of an isolated vertex $y$ together with a hairbrush $H_{n-1}$.  It remains to show that this implies $T$ must be the hairbrush $H_n$, for which it suffices to show that $x$ is adjacent to $v_{n-1}$ in $H_{n-1}$.
	
	If $x$ is adjacent to some $v_i$ or $u_i$ with $i\in [n-2]$, then $T-v_{n-1}$ has 2 components each with an even number of edges ($2n$ and 0 respectively), so $v_{n-1}\in \X$, a contradiction to $\X=\{x\}$. If $x$ is adjacent to $u_{n}$, then $T-v_{n-1}$ has 2 components each with an even number of edges ($2n-2$ and 2 respectively), so $v_{n-1}\in \tilde{X}$. Thus, in all cases, unless $x$ is adjacent to $v_{n-1}$, we have $\X\ne \{x\}$, which is a contradiction. Hence, we conclude that $T$ is $H_{n}$.   Finally, \Cref{lem:hairbrushEta} provides the other direction.
\end{proof}

\subsection{Upper bound for \texorpdfstring{$\nu(T)$}{}}
We will prove that $\nu(T)\leq (2n)!$ for all trees $T$ with $2n+1$ vertices with equality only when $T$ is the \textit{star graph} $K_{1,2n}$, which consists of a center vertex $v_0$ adjacent to $2n$ leaves $v_1,\ldots,v_{2n}$.


\begin{lemma}
	For all $n\ge 0$, $\nu(K_{1,2n})=(2n)!$.
\end{lemma}
\begin{proof}
	We prove $K_{1,2n}$ has exactly $(2n)!$ even sequences, from which the result follows by \Cref{thm:nuEta} since $K_{1,2n}$ is bipartite.  Let $\pi$ be an ordering of the vertices of $T$.  If $\pi_{2n+1}\ne v_0$, then the induced subgraph $K_{1,2n}[\pi_1,\ldots,\pi_{2n}]$ is isomorphic to $K_{1,2n-1}$ which has an odd number of edges, so $\pi$ is not an even sequence.  Hence, all even sequences $\pi$ have $\pi_{2n+1}=v_0$, and it is easy to see that every $\pi$ with this property is in fact an even sequence.
\end{proof}

We next prove some structural results regarding the set $\X$.

\begin{lemma}\label{lem:ComponentsIntersectingXOfX}
	Let $T$ be a tree with an even number of edges.
	\begin{enumerate}
		\item[(a)] No vertex $x\in \X$ is a leaf, and no two vertices of $\X$ are adjacent.
		\item[(b)] For $x\in \X$, let $T_1,\dots,T_r$, be the connected components of $T-x$ with $n_i=|V(T_i)|$. If $T_i$ contains $\ell_i$ vertices of $\X$, then $n_i\ge 2\ell_i+1$ for each $i\in [r]$.
	\end{enumerate}
\end{lemma}

\begin{proof}
	For (a),
	if $x$ is a leaf, then $T-x$ is a connected graph with an odd number of edges, so $x\notin \X$. Assume for contradiction that some $x,y\in \X$ are adjacent.  Let $T_x$ be the connected component of $T-y$ containing $x$ and similarly let $T_y$ be the component of $T-x$ containing $y$; see \Cref{fig: xset of tree}.  Since $x,y\in \X$, we have $e(T_x),e(T_y)$ even. However, it is not difficult to see  that every edge of $T-xy$ appears exactly once in either $T_x$ or $T_y$, which implies $e(T)=e(T_x)+e(T_y)+1$ is odd, a contradiction. 
	
	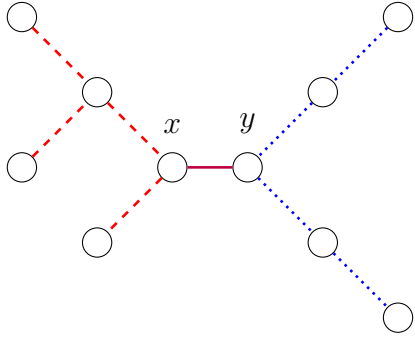
\begin{figure}[h]
		\centering
		\begin{tikzpicture}
			\node[vertexSmall] (x) at (-1,0) {};
			\node[vertexSmall] (y) at (0,0){};
			
			\node[vertexSmall] (a) at (1,1){};
			\node[vertexSmall] (b) at (1,-1){};
			\node[vertexSmall] (c) at (2,2){};
			\node[vertexSmall] (d) at (2,-2){};
			
			\node[vertexSmall] (w) at (-2,1){};
			\node[vertexSmall] (z) at (-2,-1){};
			\node[vertexSmall] (u) at (-3,0){};
			\node[vertexSmall] (v) at (-3,2){};
			
			\node[above = 3pt of x] {$x$};
			\node[above = 3pt of y] {$y$};
			
			\draw[edge,purple] (x) -- (y);
			
			\draw[edge,red,dashed] (u) -- (w) -- (x) -- (z);
			\draw[edge,red,dashed] (v) -- (w);
			
			\draw[edge,blue,dotted] (c) -- (a) -- (y) -- (b);
			\draw[edge,blue,dotted] (b) -- (d);
		\end{tikzpicture}
		\caption{A tree with $x,y\in \X$ adjacent and an odd number of edges. The tree $T_x$ is in red dashes, the tree $T_y$ is in blue dots, and the edge $xy$ is in solid purple.   }
		\label{fig: xset of tree}
	\end{figure}
	
	For (b), fix $i\in [r]$, set $\ell:=\ell_i$, and write $\{y_1,\ldots,y_\ell\}=\X\cap V(T_i)$.  For each $j\in[\ell]$, let $u_j$ be a neighbor of $y_j$ with $\mathrm{dist}_T(u_j,x)>\mathrm{dist}_T(y_j,x)$, which always exists because $y_j$ is not a leaf of $T$ by (a). Let $u_0$ be the unique neighbor of $y_0:=x$ in $T_i$.  Note that $u_j\ne y_{j'}$ for any $j,j'$ since $y_j$ and $y_{j'}$ are not adjacent by (a).  Also, observe that $u_j\ne u_{j'}$ for $j\ne j'$, as otherwise we would have $\mathrm{dist}_T(y_j,x)=\mathrm{dist}_T(y_{j'},x)$ and that $y_j,y_{j'}$ have a common neighbor $u_j=u_{j'}$ not equal to $x$, which would imply that $T$ has a cycle if $y_j\ne y_{j'}$ (via considering $u_j,y_j,y_{j'}$ and the paths from $y_j,y_{j'}$ to $x$).  In total, we conclude that $y_1,\ldots,y_\ell,u_0,u_1,\ldots,u_\ell$ are distinct vertices in $T_i$, proving the second part.
\end{proof}

We are now in position to prove our desired bound on $\nu(T)$.
\begin{prop}
	If $T$ is a tree on $2n+1$ vertices, then $\nu(T)\leq (2n)!$ with equality if and only if $T$ is the star $K_{1,2n}.$
\end{prop}

\begin{proof}
	
	We prove the result by induction on $n$, the $n=0$ case being trivial.  Fix $x\in \X$ and let $T_1,\ldots,T_r$ be the connected components of $T-x$. By induction, we have
	\begin{equation}\nu(T-x)={2n\choose n_1,\ldots,n_r} \prod_{i=1}^r \nu(T_i)\le {2n\choose n_1,\ldots,n_r} \prod_{i=1}^r (n_i-1)!=\frac{(2n)!}{\prod_{i=1}^r n_i}.\label{eq:treeMax}\end{equation}
	
	With this and \Cref{lem:ComponentsIntersectingXOfX}, we find 
	\[\nu(T-x)\le \frac{(2n)!}{\prod_{i=1}^r(2\ell_i+1)}\le \frac{(2n)!}{1+\sum_{i=1}^r 2\ell_i}=\frac{(2n)!}{1+2(|\X|-1)},\]
	where the second inequality used repeated applications of the inequality $(\alpha+1)(\beta+1)\ge \alpha+\beta+1$ valid for any $\alpha,\beta\ge 0$, and the equality used that each vertex of $\X-\{x\}$ appears in exactly one $T_i$ subtree.  Using this together with \Cref{cor:treeInductiveHypothesis} gives
	\[\nu(T)=\sum_{x\in \X} \nu(T-x)\le \frac{|\X|}{2|\X|-1} (2n)!\le (2n)!,\]
	proving the desired upper bound.  If $\nu(T)=(2n)!$, then the inequalities above must be equalities, which can only happen if $|\X|=1$, say with $\X=\{x\}$, and if $\nu(T-x)=(2n)!$.  By \eqref{eq:treeMax}, this can only happen if $n_i=1$ for all $i$, which means $T-x$ consists of $2n$ isolated vertices.  This implies $T$ is a star on $2n+1$, completing the proof.
\end{proof}

In total, this proposition together with \Cref{prop:lowerTree} completes the proof of \Cref{thm:etaTree}.

\begin{remark}Our proofs yield slightly stronger bounds on $\nu(T)$ whenever $\X$ is large.   For example, \eqref{eq:oneX} gives the lower bound $\nu(T)\ge |\X| n! 2^n$.  Bounds of this form are known as  \textit{stability results} in extremal graph theory, which roughly are results saying that bounds for a graph $T$ can be substantially improved if $T$ is ``far'' from a unique extremal construction (in our case, $T$ being ``far'' from $H_n,K_{1,2n}$ is measured by having $|\X|$ large). 
\end{remark}

\section{Multiplicity of $-1$}\label{section: multiplicity}
 In this section, we prove Proposition \ref{prop:matching} and Theorems \ref{thm: -1 multiplicity for tournaments} and \ref{thm root multiplicity upper bound} regarding the multiplicity of $-1$ as a root of $\A_{D}(t)$.  We first prove Theorem \ref{thm root multiplicity upper bound} which we restate for convenience.
\begin{theorem}
    If $D$ is an $n$-vertex digraph, then \[\mult(A_D(t),-1)\le n-s_2(n),\] where $s_2(n)$ denotes the number of 1's in the binary expansion of $n$.  Moreover, for all $n$, there exist $n$-vertex digraphs $D$ with $A_D(t)=\frac{n!}{2^{n-s_2(n)}}(1+t)^{n-s_2(n)}$.
\end{theorem}
\begin{proof}
     We first show the upper bound. Let $m$ be the multiplicity of $-1$ as a root of $\A_D(t)$. Observe that $\A_D(1) = n!$. Since there is a polynomial $p(t)$ such that $\A_D(t) = (1 + t)^mp(t)$ and $p(-1) \neq 0$, we know that $\A_D(1) = 2^mp(1) = n!$. Since $\A_D(t)$ has integer coefficients and $(1 + t)^m$ also has integer coefficients, it follows that $p(t)$ has rational coefficients. By Gauss's lemma, it follows that $p(t)$ has integer coefficients. Hence $p(1)$ is an integer. It follows that $2^m$ must divide $n!$, which by Legendre's formula implies $m\le n - s_2(n)$. 

     For the lower bound, we first consider the case when $n$ is a power of two. Let $P_2$ be the graph on vertices $v_1,v_2$ with a single arc $v_1\to v_2$. Define the sequence of digraphs $\{L_m\}_{m\in \N}$ by
     \[
        L_1 = P_2~~~~\text{and}~~~ L_{m + 1} = L_m\circ_{v_1} P_2,
     \] where we recall that this expression for $L_{m+1}$ is the rooted producted mentioned just before \Cref{prop: rooted product graphs}.  We observe that $L_{m}$ has $2^m$ vertices and $2^m - 1$ arcs. By \Cref{prop: rooted product graphs} and induction, we find
     \[
        \A_{L_m}(t) = (2^m)!\left(\frac{1 + t}{2}\right)^{2^m - 1}.
     \] Since $s_2(2^m) = 1$, this gives the desired construction when $n$ is a power of two.

     For an arbitrary $n$, let $a_1,\dots,a_\ell$ be the indices of the nonzero powers of $2$ in the binary expansion of $n$. Consider the digraph $D$ given via the disjoint union of the digraphs $L_{a_1},\dots,L_{a_\ell}$ defined previously. By Proposition \ref{prop basic facts},
     \[
        \A_{D}(t) = \binom{n}{2^{a_1}, \dots, 2^{a_\ell}}\prod_{i = 1}^{\ell} \A_{L_{a_i}}(t)=\binom{n}{2^{a_1}, \dots, 2^{a_\ell}}\prod_{i = 1}^{l} (2^{a_i})!\left(\frac{1 + t}{2}\right)^{2^{a_i} - 1}=\frac{n!}{2^{n-s_2(n)}}(1+t)^{n-s_2(n)},
     \]
     giving the desired result.
\end{proof}

We next establish our general lower bound, which we restate for convenience.
 \begin{prop}
     Let $D$ be an orientation of an $n$-vertex graph $G$.  If every matching in the complement of $G$ has size at most $m$, then  $\mult(\A_D(t),-1) \geq \floor{\frac{n}{2}} - m$.
 \end{prop}
 \begin{proof}
    We prove the result by induction on $n$, the base cases $n=0,1$ being trivial.  Assume that we have proven the result for $n-2$, and let $D$ be an orientation of an $n$-vertex graph $G$ whose complement has a maximum matching of size $m$.  By applying \Cref{lemma split into subgraphs} with $k=2$ we find
    \[
         \A_D(t) = \sum_{S\in \binom{V}{2}} \frac{t^{e_D(S,\overline{S})} + t^{e_D(\overline{S},S)}}{2}\A_{D[S]}(t)\A_{D - S}(t).
     \]
    We claim that the polynomials $\A_{D[S]}(t)\A_{D - S}(t)$ in the sum above all have $-1$ as a root with multiplicity at least $\floor{\frac{n}{2}} - m$, from which the result will follow. 

    First consider the case that $S$ is not an edge of $G$, which means it is an edge in the complement $\overline{G}$.  This implies that every maximal matching of $\overline{G}$ must use at least one vertex of $S$ (as otherwise we could include the edge $S$ into the matching), hence $\overline{G}-S$ is an $n-2$ vertex graph with maximum matching of size at most $m-1$.  Inductively this implies $\A_{D - S}(t)$ has $-1$ as a root with multiplicity at least $\floor{\frac{n-2}{2}} - (m-1)=\floor{\frac{n}{2}} - m$, giving the desired result.
    
    Next consider the case that $S$ is an edge in $G$.  Observe that $\overline{G}-S$ is an $n-2$ vertex graph which continues to have no matching of size larger than $m$, so inductively $A_{D-S}(t)$ has $-1$ as a root with multiplicity at least $\floor{\frac{n-2}{2}} - m=\floor{\frac{n}{2}} - m-1$.  Also note that since $S$ is an edge, $A_{D[S]}(t)=1+t$, so combining these two terms gives the desired multiplicity.  This completes the proof of the claim, proving the result.
 \end{proof}
In particular this result implies $\mult(A_D(t),-1)\ge \floor{\frac{n}{2}}$ for tournaments $D$, but proving this holds with equality requires a refinement of \Cref{lemma split into subgraphs} which requires some additional notation.

  Let $OP(\alpha)$ denote the set of all ordered set patitions of type $\alpha$, and let $SP(\lambda)$ denote the set of all unordered set partitions with type $\lambda$.  For a digraph $D$ and an ordered set partition $P$ of the vertices of $D$ of length $k$ and $i \in [k]$, define the \textit{$i$-th forward sequence number} of $P$  to be \[FS_{D,P}(i) = \sum_{j = i + 1}^{k} e_D(P_i,P_j)\] and the \textit{$i$-th reverse sequence number} of $P$ to be \[RS_{D,P}(i) = \sum_{j = i + 1}^{k} e_D(P_j,P_i)\] where we set $FS_{D,P}(k) = 0$ and $RS_{D,P}(1) = 0$. Note that $FS_{D,P}(i)=e_D(P_i,P_{i+1}\cup \cdots P_k)$ and $RS_{D,P}(i)=e_D(P_{i+1}\cup \cdots P_k,P_i)$. 

  With this notation in hand, we have the following corollary of Lemma \ref{lemma split into subgraphs}.
  \begin{lemma}\label{lemma ordered set partition expansion}
    If $D$ is a digraph on the vertex set $[n]$ and $\alpha$ is an integer composition of $n$ of length $k$, then
    \[
        \A_D(t) = \frac{1}{2^{k}}\sum_{P\in OP(\alpha)} \prod_{i = 1}^{k} \left(\A_{D[P_i]}(t) \left(t^{FS_{D,P}(i)} + t^{RS_{D,P}(i)}\right)\right) 
    \]
 \end{lemma}
 \begin{proof}
     We induct on the number of parts in $\alpha$. If $\alpha$ has one part, then the only partition $P$ is the entire vertex set. Hence, $FS_{D,P}(i) = 0 = RS_{D,P}(i)$ for every $i\in [k]$ and the result follows.

     Assume the claim holds for any integer composition with $k - 1$ parts, and consider an integer partition $\alpha$ with $k$ parts. Let $\alpha_1$ be the first part of $\alpha$, and let $\alpha'$ be the integer composition given by removing the first part of $\alpha$. By Lemma \ref{lemma split into subgraphs}, 
     \begin{equation}
        \A_{D}(t) = \sum_{S\in\binom{[n]}{\alpha_1}} \frac{t^{e_D(S,\overline{S})} + t^{e_D(\overline{S},S)}}{2}\A_{D[S]}(t)\A_{D - S}(t).\label{eq:partition1}
     \end{equation} By our inductive hypothesis, for each $S\in \binom{[n]}{\alpha_1}$ 
      \begin{equation}\label{eq:partition2}
        \A_{D - S}(t) = \frac{1}{2^{k - 1}}\sum_{P\in OP(\alpha')} \prod_{i = 2}^{k} \left(\A_{(D - S)[P_i]}(t) \left(t^{FS_{D - S,P}(i)} + t^{RS_{D - S,P}(i)}\right)\right) 
    \end{equation} Observe that for $i\in \{2,\ldots,k\}$ we have
    \[
        FS_{D,P}(i) = FS_{D - S,P}(i) ~~\text{and}~~RS_{D,P}(i) = RS_{D - S,P}(i)
    \] and that
    \[
        FS_{D,P}(1) = e_D(S,\overline{S})~~\text{and}~~RS_{D,P}(1) = e_D(\overline{S},S).
    \] It follows from \eqref{eq:partition1} and \eqref{eq:partition2} that
    \[
        \A_D(t) = \frac{1}{2^{k}}\sum_{P\in OP(\alpha)} \prod_{i = 1}^{k} \left(\A_{D[P_i]}(t) \left(t^{FS_{D,P}(i)} + t^{RS_{D,P}(i)}\right)\right) 
    \] as desired.
 \end{proof}

With the lemma, we can now prove Theorem \ref{thm: -1 multiplicity for tournaments}. We restate the theorem for convenience.
\begin{theorem}\label{theorem -1 multiplcity for even tournaments}
     If $D$ is a tournament on $n$ vertices, then $\mult(\A_D(t),-1) = \lfloor \frac{n}{2}\rfloor$.
\end{theorem}
\begin{proof}
    We first consider the case when $n = 2k$ for some $k\in \N$. By Proposition \ref{prop:matching}, $\mult(\A_D(t),-1) \geq n/2$. By Lemma \ref{lemma ordered set partition expansion} applied to the partition $(2)^k$, we have
    \[
        \A_D(t) = \left(1 + t\right)^k\frac{1}{2^{k}}\sum_{P\in OP((2)^k)} t^{FS_{T,P}(i)} + t^{RS_{T,P}(i)},
    \] where here we used that $A_{D[P_i]}(t)=1+t$ for all sets $P_i$ of size 2 since $D$ is a tournament.  Let \[p(t) = \frac{1}{2^{k}}\sum_{P\in OP((2)^k)}t^{FS_{T,P}(i)} + t^{RS_{T,P}(i)}.\] We claim that $p(-1) \neq 0$, from which the bound $\mult(A_D(t),-1)\le n/2$ will follow from the inequality above.
    
    We first observe that $FS_{T,P}(i)+ RS_{T,P}(i)$ is always even, as both vertices in $P_i$ are adjacent to every vertex in $P_{i + 1}\cup \dots \cup P_{k}$. Thus for every $P\in OP((2)^k)$ and every $i\in [k]$, \[(-1)^{FS_{T,P}(i)}+ (-1)^{ RS_{T,P}(i)}  = 2(-1)^{FS_{T,P}(i)}.\]  It follows that 
    \[
        p(-1) = \frac{1}{2^{k}}\sum_{P\in OP((2)^k)} \prod_{i = 1}^{k}(2(-1)^{FS_{T,P}(i)}) = \sum_{P\in OP((2)^k)} (-1)^{FS_{T,P}} = \sum_{P\in SP((2)^k)} \sum_{\sigma\in \mathfrak{S}_k} (-1)^{FS_{T,\sigma P}}.
    \] where $\sigma P$ denotes the ordered set partition $(P_{\sigma(1)},\dots,P_{\sigma(k)})$. We first establish the following claim:
    \begin{claim}\label{eq: fs mod 2 is same for all rearrangements of P}
        For all $\sigma\in \mathfrak{S}_k$,
        \begin{equation}
    FS_{T,\sigma P}\equiv_2 FS_{T,P}~~\text{and}~~ RS_{T,\sigma P}\equiv_2 RS_{T,P}.\end{equation}
    \end{claim}
    \begin{proof}
         It suffices to consider the cases where
        \[P=(P_1,\dots,P_{a-1},P_a,P_{a+1},P_{a+2},\dots,P_k)\qquad P'=(P_1,\dots,P_{a-1},P_{a+1},P_{a},P_{a+2},\dots,P_k)\]
        for some $a\in [k-1]$. We then have 
        \begin{align*}
            {FS}_{D,P} -  {FS}_{D,P'} &= \sum_{i = 1}^{k} {FS}_{D,P}(i) - {FS}_{D,P'}(i)\\
            &= \sum_{i = 1}^{k} \sum_{j = i + 1}^{k} ({e}_D(P_i,P_j) - {e}_D(P_i',P_j'))\\
            &= ({e}_D(P_a,P_{a + 1}) - {e}_D(P_a',P_{a + 1}'))\\
            &= {e}_D(P_a,P_{a + 1}) - {e}_D(P_{a + 1},P_{a})\\
            &\equiv_2 0
        \end{align*} since ${e}_D(P_a,P_{a + 1}) + {e}_D(P_{a + 1},P_{a})$ is the number of (undirected) edges from $P_a$ to $P_{a + 1}\cup \dots \cup P_k$, which is always even. The result for $RS_{T,\sigma P}$ follows by an identical argument.
    \end{proof}

    With claim \eqref{eq: fs mod 2 is same for all rearrangements of P} and the fact that $|OP(2^k)|=k!|SP(2^k)|$, we can conclude that 
    \[
        p(-1)= k!\sum_{P\in SP(2^k)}(-1)^{FS_{T,P}}.
    \]
    Each $P\in SP((2)^k)$ is a perfect matching on $[2k]$, and there are $(2k - 1)!!$ such perfect matchings. Since $(2k - 1)!!$ is odd, it follows that $p(-1) \neq 0$ as desired. 
    
    Now when $n = 2k + 1$, we can apply the same reasoning as in the above proof to the integer composition $(2^k,1)$. The conclusion follows from the fact that there are $(2k - 1)!!\cdot (2k + 1)$ maximum matchings in $K_{2k + 1}$.
\end{proof}

\section{Concluding Remarks and Open Problems}\label{sec conclusion}
In this paper we studied a notion of Eulerian polynomials $\A_D(t)$ for digraphs $D$ and proved a number of results related to evaluations at $t=-1$.  We conclude the paper by listing a number of remaining open problems themed around interpreting $\nu(G)$ and multiplicities of $-1$ as a root of $\A_D(t)$.

\textbf{Interpretations for $\nu(G)$}.  Recall that for any graph $G$ we define $\nu(G)=|\A_D(-1)|$ where $D$ is any orientation of $G$.  While \Cref{thm:nuEta} provides a combinatorial interpretation for $\nu(G)$ when $G$ is bipartite, complete bipartite, or a blowup of a cycle, we are still far from understanding this quantity for general graphs, which we leave as the main open problem for this paper.

\begin{question}\label{quest:combInterp}
    Can one give a combinatorial interpretation for $\nu(G)$ for arbitrary graphs $G$?
\end{question}
In view of \Cref{thm:nuEta} and the bound $\nu(G)\le \eta(G)$ from \Cref{prop:nuGeneral}(a), we suspect that in general $\nu(G)$ should count even sequences of $G$ with some special properties, but what these properties should be remains a mystery.

To answer \Cref{quest:combInterp}, it might be useful to establish which graphs $G$ satisfy $\nu(G)=\sum_v \nu(G-v)$, as recurrences of this form were a key step in proving \Cref{thm:nuEta}.   In particular, computational evidence suggests that the following could hold, where here we recall that a graph is \textit{Eulerian} if all of its degrees are even.
\begin{conjecture}\label{conj:Eulerian}
    If $G$ is an Eulerian graph, then $\nu(G)=\sum_v \nu(G-v)$.
\end{conjecture}
We note that Eulerian graphs have a ``natural'' orientation via orienting each edge according to an Eulerian tour.  Given that e.g.\ our proof of \Cref{cor:nuBipartite} relied on ``natural'' orientations of bipartite graphs, it is plausible that this natural orientation for Eulerian graphs could be used to prove \Cref{conj:Eulerian}.

Our proof of \Cref{thm:nuEta} is non-combinatorial, and it would be interesting to have a more direct combinatorial proof of this fact, say for bipartite graphs.

\begin{prob}
    For any bipartite graph $G=([n],E)$ and orientation $D$ of $G$, construct an explicit involution $\phi:\mathfrak{S}_n\to \mathfrak{S}_n$ such that
    \begin{enumerate}
        \item[(a)] The set of fixed points $\mathcal{F}_\phi$ of $\phi$ is the set of (inverses of) even sequences of $G$, and
        \item[(b)] $(-1)^{\des_D(\sig)}=-(-1)^{\des_D(\phi(\sig))}$ for all $\sig\notin \mathcal{F}_\phi$.
    \end{enumerate}
\end{prob}

Such an involution is known to exist when $G=P_n$ (i.e.\ when inverses of even sequences are exactly alternating permutations), but this involution is somewhat complex; see \cite[Exercise 135]{stanley_2011} for more.

\textbf{Multiplicity of Roots}.  In \Cref{thm: -1 multiplicity for tournaments} we showed every $n$ vertex tournament $D$ has $-1$ as a root of $\A_D(t)$ with multiplicity exactly $\floor{\frac{n}{2}}$. A natural generalization of this result would be the following. 
\begin{conjecture}\label{conj:multipartite}
    If $D$ is the orientation of a complete multipartite graph which has $r$ parts of odd size, then $\mult(\A_D(t),-1)=\floor{\frac{r}{2}}$.
\end{conjecture}
Observe that the bound $\mult(A_D(t),-1)\ge \floor{\frac{r}{2}}$ follows from \Cref{prop:matching}, so the difficulty lies in proving the upper bound.

Another direction is to look at the more general quantity $\mult(A_D(t),\alpha)$, which is defined to be the multiplicity of $\alpha$ as a root of $A_D(t)$.  For example, it is not difficult to see that $\mult(A_D(t),0)$ is equal to the minimum number of arcs that one must remove in $D$ to obtain an acyclic digraph.  Such a set of arcs is known as a \textit{minimum feedback arc set}, and determining the size of such a set is well known to an NP-Complete problem \cite{karp2010reducibility}.

This connection to feedback arc sets, together with the result of this paper, establishes a number of results for $\mult(A_D(t),\alpha)$ when $\alpha\in \{0,-1\}$, and it is natural to ask what can be said about other integral values of $\alpha$.  An immediate obstacle to this is the following.
\begin{question}
    Does there exist a digraph $D$ such that $A_D(t)$ has an integral root which is not equal to either $0$ or $-1$?
\end{question}
We have verified that no such digraph exists on at most 5 vertices. 
We also note that there exist digraphs with real roots of magnitude larger than $2$, so the obstruction to finding these integral roots is not that their magnitudes are too large.

\subsection*{Acknowledgement} This work began as part of the Graduate Student Combinatorics Conference 2022, which was funded through NSF Grant DMS-1933360, UC San Diego, and the Combinatorics Foundation. 

    \bibliographystyle{amsalpha}
\bibliography{refs}

\end{document}